\documentclass[11pt,oneside,notitlepage]{article}
\pdfoutput=1
\usepackage{amssymb,latexsym,amsmath}
\usepackage{amsfonts,amsthm}
\usepackage{adjustbox}
\usepackage{amssymb,amscd,pxfonts}
\usepackage{amsrefs}
\usepackage{caption}
\usepackage{subcaption}
\usepackage{fancyhdr}
\usepackage{placeins}
\usepackage{graphicx}
\usepackage{hyperref}
\usepackage{mathtools}
\usepackage{caption,subcaption}
\usepackage{tikz}
\usepackage{pgfplots}
\usepackage{txfonts}
\usepackage[normalem]{ulem}
\usepackage{geometry}
\usepackage{booktabs}
\mathtoolsset{showonlyrefs}

\newcommand{\F}{\mathcal{F}}

\newcommand{\ephi}{\mathcal{E}_{\phi}^\xi}

\newcommand{\ez}{\mathcal{E}_0^\xi}

\newcommand{\norm}[1]{\left\|#1\right\|}

\newcommand{\A}{{\cal{A}}}

\newcommand{\R}{\mathbb{R}}

\newcommand{\tphil}{\mathbb{T}_{\phi L}}

\newcommand{\uo}{u_{\omega,\phi}}

\newcommand{\um}{u_{m,\phi}}
\newcommand{\us}{u_{s,\phi}}

\newtheorem{thm}{Theorem}[section]

\newtheorem{definition}[thm]{Definition}

\newtheorem{lemma}[thm]{Lemma}
\newtheorem{notation}[thm]{Notation}

\newtheorem{remark}[thm]{Remark}
\newtheorem{theorem}[thm]{Theorem}
\numberwithin{equation}{section}   
\numberwithin{thm}{section}
\numberwithin{figure}{section}

\def\Xint#1{\mathchoice
{\XXint\displaystyle\textstyle{#1}}%
{\XXint\textstyle\scriptstyle{#1}}%
{\XXint\scriptstyle\scriptscriptstyle{#1}}%
{\XXint\scriptscriptstyle\scriptscriptstyle{#1}}%
\!\int}
\def\XXint#1#2#3{{\setbox0=\hbox{$#1{#2#3}{\int}$}
\vcenter{\hbox{$#2#3$}}\kern-.5\wd0}}

\def\dashint{\Xint-}
\begin{document}
\title{Numerics and analysis of Cahn--Hilliard critical points}

\author{Tobias Grafke\footnote{University of Warwick},
Sebastian Scholtes\footnote{Technical University of Applied Sciences Augsburg},
Alfred Wagner\footnote{RWTH Aachen University}
and Maria G. Westdickenberg$^{\ddagger,}$
} \date{\today}
\maketitle

\begin{abstract}
We explore recent progress and open questions concerning local minima and saddle points of the Cahn--Hilliard energy in $d\geq 2$ and the critical parameter regime of large system size and mean value close to $-1$. We employ the String Method of E, Ren, and Vanden-Eijnden---a numerical algorithm for computing transition pathways in complex systems---in $d=2$ to gain additional insight into the properties of the minima and saddle point. Motivated by the numerical observations, we employ a method of Caffarelli and Spruck to study convexity of level sets in $d\geq 2$.
\end{abstract}


\section{Introduction}\label{S:intro}

The Cahn--Hilliard energy is a fundamental model of phase separation, with wide-ranging applications in materials science and other natural sciences. Introduced by Cahn and Hilliard to capture the mixing of a binary alloy, both the energy and the dynamics have been the subject of extensive numerical and analytical studies; we refer for instance to \cites{CH,CH3,NS} and the articles cited therein. Despite the large body of existing work on the model, significant open questions remain.

In this paper we analyze critical points of the renormalized energy
\begin{align}
\ephi(u)&=\int_{\tphil}\frac{\phi}{2}|\nabla u|^2+\frac{1}{\phi}\Big(G(u)-G(-1+\phi)\Big)\,dx\label{ephi}
\end{align}
on
\begin{align}\label{H-phi}
X_\phi:=\Big\{u\in H^1\cap L^4(\tphil): \dashint_{\tphil} u \, dx=-1+\phi\Big\},
\end{align}
where $\tphil$ is the $d$-dimensional torus with side length $\phi L$. Here $G$ is a nondegenerate double-well potential with minima $\pm 1$ (the standard choice being $G(u)=\frac14(1-u^2)^2$); see Remark~\ref{rem:G} below. We are interested in the so-called critical regime, that is, the regime such that
\begin{align*}
L\sim\phi^{-(d+1)/d} \quad\text{and}\quad\phi\ll 1.
\end{align*}
We recall the origins of the problem in Section \ref{S:detail} below.

In \cite{GWW} the existence of a nonconstant \emph{local energy minimizer} and \emph{saddle point} of the energy in the critical regime was established. Moreover, the following \emph{constrained energy minimizers} were introduced and analyzed:
functions $\uo$ that minimize $\ephi$ subject to
\begin{align}
 \uo\in X_\phi\quad\text{and}\quad \nu(\uo)=\omega,\label{vol}
\end{align}
where the ``volume'' of a function $u\in X_\phi$ is defined as
\begin{align}
\nu(u):=\int_{\tphil} \chi(u)\,dx,\label{nu}
\end{align}
for a smooth function $\chi$ such that $\chi(s)=0$ for  $s\leq 1-2\phi^{1/3}$ and $\chi(s)=1$ for  $s\geq 1-\phi^{1/3}$. Since it can be shown that volume-constrained minimizers are bounded above by $1+\phi^{1/3}$, $\nu$ roughly measures the volume $|\{u\approx 1\}|$.

While much is known about the shape of minimizers on $\R^d$ (see for instance \cite{GNN,BL,LN}), the shape of minimizers on bounded domains is less well understood; as seminal works on the torus, we mention \cite{K,Br}. The Cahn-Hilliard problem on the torus features a competition between the periodic boundary conditions and the energetic preference for spherical level sets. We are interested in \emph{understanding and quantifying this competition}. Note that this issue is more than a matter of mathematical/geometric curiosity; the appropriateness of the often-used sphericity assumption is of interest in nucleation theory in physics and engineering applications; see for instance \cite{L}.

Quantitative estimates
on the energy and shape of critical points were derived in \cites{GWW,GWW20}. For example the Fraenkel asymmetry of the superlevel sets of $\uo$ is shown to be bounded by $\phi^{1/6}$ \cite{GWW}; see Section \ref{S:detail} below for more detail. In addition, the volume-constrained minimizers are shown to be equal to their Steiner symmetrization \cite{GWW20} and hence have connected superlevel sets. It is \emph{plausible} that the saddle point (whose existence was pointed out already by Cahn and Hilliard in \cite{CH3} and is established via a mountain-pass argument in \cite{GWW}) is in fact equal to a volume-constrained minimizer of appropriate volume, however this has not yet been proved. Hence although bounds are obtained for the constrained minimizers, \cites{GWW,GWW20} stop short of establishing such bounds for the actual saddle point.

Given the open questions in the analysis, in this article we apply the String Method \cite{ERV} to look numerically at the local minimum and saddle point of the Cahn--Hilliard energy in $d=2$ in Section~\ref{sec:numerical-results}.
The String Method is a well-known algorithm for computing transition pathways in complex systems. Given two local minimizers of an energy, it locates a saddle point on the boundary of the corresponding basins of attraction. Alternatively, given two arbitrary states in distinct basins of attraction, the string evolves to find both the local minima of the basins and a saddle point between them. We review the method in Subsection \ref{ss:string} and summarize our numerical results in Subsection \ref{ss:numer}. Then in Section \ref{S:convex}, motivated in part by our numerical observations, we study convexity of the sub/superlevel sets.
In particular, the results of the String Method suggest that for both the minimum and saddle point there is a value $t_*\in(-1,1)$ such that
\begin{align*}
\text{the superlevel sets }\{u>t\}\text{ for $t\geq t_*$ and the sublevel sets }\{u<t\}\text{ for $t\leq t_*$}
\end{align*}
are convex.
While proving such a fact is currently out of reach, we show how the method of Caffarelli and Spruck \cite{CS} can be applied to obtain a partial result in this direction; see Proposition~\ref{convex5new}.

\begin{remark}\label{rem:G}
We assume that the double-well potential $G$ is smooth and satisfies
\begin{itemize}
  \item $G> 0$ for $x\neq \pm 1$ and $G(\pm 1)=0$,
  \item $G'\geq 0$ on $[-1,0]$, $G'\leq 0$ on $[0,1]$, and $G''(\pm 1)>0$.
\end{itemize}
\end{remark}
\begin{notation}
  For two vectors $v$ and $w$ and a Hessian matrix $A$, we will use the notation
$v\cdot A\cdot w=\sum_{i,j=1}^{n}v_{i}\,A_{ij}\,w_{j}$. Similarly, for a tensor $A=A_{ijk}$, we use $u\cdot (v\cdot A\cdot w)$ to denote:
\begin{align*}
  u\cdot (v\cdot A\cdot w)=\sum_{i,j,k=1}^{n}u_k v_i A_{i,j,k} w_j.
\end{align*}
\end{notation}

\section{Background and  open questions}\label{S:detail}

The original Cahn--Hilliard energy landscape consists of the energy
\begin{align}
 E(u):=\int_\Omega \frac{1}{2}|\nabla u|^2 + G(u) \,dx\label{E}
\end{align}
considered on the set
\begin{align}\notag
\Big\{u\in H^1\cap L^4(\Omega): \dashint_\Omega u \, dx=m\Big\},
\end{align}
where the mean value $m$ is a given value strictly between the minima of $G$. Studying the problem on a large domain $|\Omega|$ is equivalent, after a rescaling of space, to studying the energy
\begin{align*}
 E_\phi(u):=\int_{\Omega_1}\frac{\phi}{2}|\nabla u|^2+\frac{1}{\phi}G(u)\,dx
\end{align*}
for $\phi\ll 1$ on the fixed domain $\Omega_1$. It is well-known since the seminal work of Modica and Mortola \cite{MM} that for $m$ fixed and $\phi\downarrow 0$, the energy $E_\phi$ $\Gamma$-converges to the perimeter functional.

Droplet formation or nucleation in a low-density phase is more subtle, as is pointed out by Biskup, Chayes, and Kotecky in \cite{BCK}.
For the Cahn-Hilliard model, \cite{BGLN,CCELM}
probed the \emph{competition} between large system size and mean value close to $-1$. To fix ideas, consider the energy \eqref{E} with $\Omega=\mathbb{T}_L$, the torus of side-length $L$. It is shown in \cite{BGLN,CCELM} that mean value
\begin{align*}
 m:=-1+\phi
\end{align*}
with $\phi$ and $L$ scaling like
\begin{align*}
 \phi\sim L^{-d/(d+1)}
\end{align*}
is the \emph{critical parameter regime} in the following sense: There is a dimension-dependent constant $\xi_d>0$ such that for $\xi<\xi_d$ and
\begin{align*}
  \phi=\xi L^{-d/(d+1)},
\end{align*}
  \emph{the global minimizer} is the constant state $u\equiv -1+\phi$, while for $\xi>\xi_d$, \emph{the global minimizer} is \textbf{not} the constant state and is instead close in $L^p$ to a ``droplet-shaped'' function.

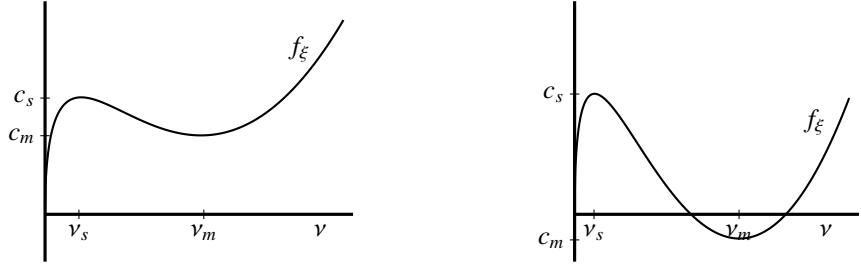
\begin{figure}
\begin{tikzpicture}[scale=0.63]
\hspace{1cm}
\centering
\draw [black, very thick] (5,0) -- (-1.5,0);
\draw [black, very thick] (-1.5,-1) -- (-1.5,4.5);
\draw[black, thick, domain=-1.5:4.8, samples=500] plot [smooth] (\x, {9.7*sqrt(0.4*(\x+1.5))-10.15*0.4*(\x+1.5)+2*0.18*(\x+1.5)*(\x+1.5)});
\node [below] at (1.85,0) {\small{$\nu_m$}};
\node [below] at (4.3,0) {\small{$\nu$}};
\node [left] at (-1.5,1.66) {\small{$c_m$}};
\node [left] at (-1.5,2.455) {\small{$c_s$}};
\draw (-1.6,2.455) -- (-1.4,2.455);
\draw (-1.6,1.66) -- (-1.4,1.66);
\node [left] at (4.3,3.5) {\small{$f_\xi$}};
\draw (2.35-0.5,0.1) -- (2.35-0.5,-0.1);
\draw (-0.285-0.5,0.1) -- (-0.285-0.5,-0.1);
\node [below] at (-0.285-0.5,0) {\small{$\nu_s$}};
\hspace{2cm}
\draw [black, very thick] (12+0.5,0) -- (6+0.5,0);
\draw [black, very thick] (6+0.5,-1) -- (6+0.5,4.5);
\draw[black, thick, domain=6+0.5:11.8+0.5, samples=500] plot [smooth] (\x, {15*sqrt(0.3*(\x-6-0.5))-23*0.3*(\x-6-0.5)+7.5*0.09*(\x-6-0.5)*(\x-6-0.5)});
\node [left] at (11.5+0.5,1.95) {\small{$f_\xi$}};
\draw (9.47+0.5,0.1) -- (9.47+0.5,-0.1);
\node [below] at (9.47+0.5,0) {\small{$\nu_m$}};
\draw (5.9+0.5,2.54) -- (6.1+0.5,2.54);
\draw (6.41+0.5,-0.1) -- (6.41+0.5,0.1);
\draw (5.9+0.5,-0.54) -- (6.1+0.5,-0.54);
\node [left] at (6+0.5,-0.54) {\small{$c_m$}};
\node [left] at (6+0.5,2.54) {\small{$c_s$}};
\node [below] at (11.3+0.5,0) {\small{$\nu$}};
\node [below] at (6.41+0.5,0) {\small{$\nu_s$}};
\end{tikzpicture}
\caption{The function $f_\xi$ defined in \eqref{f} has for $\tilde{\xi}_d<\xi<\xi_d$ a local (but not global) minimum at $\nu_m>0$ (left figure) and for  $\xi>\xi_d$ a global minimum at $\nu_m>0$ (right figure). For $\xi<\tilde{\xi}_d$, the only local minimum is at zero (not shown).}\label{fig:f}
\end{figure}

Comparing the energy \eqref{E} of a function $u$ on a large torus to the energy of the constant state and renormalizing leads, after a rescaling of space, to the energy \eqref{ephi} (see \cite{GW}).
The $\Gamma$-convergence of \eqref{ephi}
to
\begin{align}
\ez(u):=
c_0{\rm Per}(\{u=1\})-4|\{u=1\}|+4\frac{|\{u=1\}|^2}{\xi^{3}}\notag
\end{align}
in the critical regime is established in \cite{GW}. This limit energy is of the type predicted by classical nucleation theory (going back to Gibbs \cite{G}): A positive term penalizes the surface energy of the region of positive phase, while a volume term reflects the reduction in bulk energy when the volume of the positive phase is large enough. (The higher order positive term prevents the positive phase from growing too large and is a ``remnant'' of the mean constraint.)

When the limit energy $\ez$ is minimized as a function of the volume $|\{u=1\}|$, one obtains the function $f_\xi:[0,\infty)\to\R$ defined by
\begin{align}
 f_\xi(\nu):=\bar{C}_1\nu^{(d-1)/d}-4\nu+4\xi^{-(d+1)}\nu^2,\quad \text{for}\quad \bar{C}_1:=c_0\sigma_d^{1/d}d^{(d-1)/d},\label{f}
\end{align}
where $c_0$ is the minimum energy cost for an interface connecting $\pm 1$ (equal to $2\sqrt{2}/3$ for the standard potential) and $\sigma_d$ is the surface area of the unit sphere in $\R^d$. The behavior of $f_\xi$ and its dependence on $\xi$ is illustrated in Figure \ref{fig:f}.
For $\xi>\xi_d$, the strictly positive local minimum of $f_\xi$ is the global minimum. This corresponds to the droplet-shaped global minimizer of $\ephi$ captured in \cite{BGLN,CCELM,GW}.
For $\tilde{\xi}_d<\xi<\xi_d$ (see \cite{GWW} for an explicit computation of the bifurcation point $\tilde{\xi}_d$), the function $f_\xi$ has a global minimum at zero and a strictly positive local minimum at $\nu_m$ with value $c_m$. This corresponds to the droplet-shaped local minimizer of $\ephi$ whose existence is proved in \cite{GWW}. The other volume-constrained minimizers of $\ephi$ subject to \eqref{vol} are also connected to the graph of $f_\xi$; see Subsection \ref{ss:quant}.

The function $f_\xi$ has a local \emph{maximum} at $\nu_s\in (0,\nu_m)$ with value $c_s$. It is natural to conjecture that this fact should indicate that there is a saddle point of $\ephi$ with volume approximately $\nu_s$. For both $\xi>\xi_d$ and  $\tilde{\xi}_d<\xi\leq\xi_d$, a mountain pass argument leads to the existence of a saddle point of $\ephi$ in between the constant state and the droplet-shaped minimizer \cite{GW,GWW}. One expects that this saddle point is the ``critical nucleus'' predicted by Cahn and Hilliard \cite{CH3} and that it is also approximately droplet-shaped, like the minimum. Unfortunately the mountain pass argument fails to provide much information about the saddle point.  While it does yield that the \emph{energy} of the saddle point $u_s$ is close to $c_s$, it yields no information about the measure of $\{u_s\approx 1\}$  or about the shape or properties of this set. This gap in the analysis motivates this paper.

\subsection{Quantitative bounds}\label{ss:quant}
In \cite{GWW}, the following bounds were developed for the local minimizer $\um$ and volume-constrained minimizers $\uo$.
The local minimum $\um$ has energy and ``volume'' close to $c_m$ and $\nu_m$ with the bounds
\begin{align}
 \left|\ephi(\um)-c_m\right|&\lesssim \phi^{1/3},\label{emin}\\
   |\nu(\um)-\nu_m|&\lesssim C(\xi)\phi^{1/6}.\label{num}
\end{align}
Here and throughout, $\lesssim$ is used to indicate an upper bound by a constant depending at most on the dimension and the potential $G$. In addition, $\um$ is estimated in $L^2$ against the function $\Psi_m$ that is $+1$ in a ball of volume $\nu_m$ and $-1$ otherwise according to
\begin{align}
  |\um-\Psi_m|_{\tphil}^2\lesssim
  \begin{cases}\label{visa}
   C(\xi) \phi^{1/6}&\text{for}\quad d=2,3\\
  \phi^{1/(2d)}&\text{for}\quad d\geq 4,
  \end{cases}
\end{align}
where
\begin{align*}
 |u-\Psi_m|_{\tphil}&:=\inf_{x_0\in\tphil}||u-\Psi_m(\cdot-x_0)||_{L^2(\tphil)},
 \end{align*}
and $\Psi_m$ is understood in the periodic sense. In addition via elementary arguments one obtains the following bound on the volume of the interfacial region:
\begin{align}
  \lvert\{-1+2\phi^{\frac{1}{3}}\leq u_{m,\phi}\leq 1-2\phi^{\frac{1}{3}}\}\rvert\lesssim \phi^{1/3}.\label{Imin}
\end{align}
The analogous results hold for $\uo$ with volume $\omega$ and energy $f_\xi(\omega)$. Our numerical investigation in Subsection \ref{ss:numer} is in part intended to explore whether these exponents are sharp.

Similarly, let $\tilde{u}_s$ denote any volume-constrained minimum with energy $c_s+O(\phi^{1/3})$. Then we have
\begin{align}
|\ephi(\tilde{u}_s)-c_s|&\lesssim \phi^{1/3},\label{enps}\\
  | \nu(\tilde{u}_s)-\nu_s|&\lesssim C(\xi) \phi^{1/6},\label{wmm2}
  \end{align}
  and
  \begin{align}
  |\tilde{u}_s-\Psi_s|_{\tphil}^2\lesssim\label{ven2}
  \begin{cases}
  C(\xi) \phi^{1/6}&\text{for}\quad d=2,3\\
  \phi^{1/(2d)}&\text{for}\quad d\geq 4,
  \end{cases}
\end{align}
where $\Psi_s$ is $+1$ in a ball of volume $\nu_s$ and $-1$ otherwise. The \emph{saddle point} $\us$ obeys \eqref{Imin} and\eqref{enps}. We present numerical results in Subsection \ref{ss:numer} that explore whether it also satisfies \eqref{wmm2} and \eqref{ven2}.

In \cite{GWW20} it is also shown that the constrained minima are  \emph{equal to their Steiner symmetrization} with respect to some point in the torus. This yields in particular connectedness of appropriate superlevel and sublevel sets. Open questions include whether the constrained minima are unique and whether their level sets are convex. The numerical results in Subsection \ref{ss:numer} support these conjectures. A partial result on convexity of superlevel sets is given in Proposition~\ref{convex5new}.

\section{Numerical results}
\label{sec:numerical-results}

In this section we  compute the minimum and saddle point
configurations numerically for a range of values of $\phi$ in order to compare
against the scaling predictions outlined above.
To find
the minimum, we can simply start from a localized approximate droplet shape and
relax via the Cahn-Hilliard dynamics until they become (approximately) stationary. In order to find the saddle points without
inserting any prior knowledge of the system, we will compute
the heteroclinic orbit connecting the spatially homogeneous solution
to the droplet minimum. This heteroclinic orbit must necessarily pass
through a saddle point in-between, which can then easily be
identified. In order to compute this heteroclinic orbit, we employ the
String Method~\cite{ERV}.

There are other algorithms to compute stable and unstable
fixed points of dynamical systems, including  gentlest ascent
dynamics~\cite{e-zhou:2011}, Newton methods, and edge tracking and
dimer-type algorithms~\cite{gould-ortner-packwood:2016,
  levitt-ortner:2017}. Here, we use the String Method despite its
comparable computational inefficiency because (1) it does not require
higher derivatives of the system, (2) it does not rely on approximate
knowledge of the saddle point configuration, and (3) is very simple to
implement given an existing code integrating the dynamics.

\subsection{The String Method}
\label{ss:string}

The String Method starts with a collection of states in phase space (a
``string'' of states). In the first step, one evolves each of the
states forward in time. In the second step, one reparameterizes the
string to keep the states well-separated in phase space. As an
example, consider the gradient flow
\begin{align*}
  \dot x(t)=-\nabla V(x).
\end{align*}
Suppose that $x_*$ and $x_{**}$ are two isolated minima of $V$. We
want to find the orbit connecting $x_*$ and $x_{**}$, i.e., the curve
$\gamma$ connecting $x_*$ and $x_{**}$ such that
\begin{align*}
  (\nabla V)^\perp (\gamma)=0,
\end{align*}
where $(\nabla V)^\perp$ is the component of $\nabla V$ normal to
$\gamma$:
\begin{align*}
  (\nabla V)^\perp (\gamma)=\nabla V(\gamma)-(\nabla V(\gamma),\hat \tau)\hat \tau.
\end{align*}
Here $\hat \tau$ represents the unit tangent vector to the curve
$\gamma$ and $(\cdot,\cdot)$ represents the Euclidean inner
product. Note that the orbit will pass through at least one additional
critical point, the saddle point that we seek.

The first step is to choose a method of parameterizing paths
$\{\phi(\alpha)\colon \alpha\in [0,1]\}$. Parameterization by
arclength is a convenient choice (although other methods offer
advantages; see \cite{ERV2}). Given any initial path $\gamma_0$
connecting $x_*$ and $x_{**}$, we will generate the family of
time-dependent paths $\phi(\alpha,t)$ in the following way. In the
original String Method, the evolution is prescribed as
\begin{align*}
  \dot\phi=-\nabla V(\phi)^\perp +\lambda\hat \tau,\qquad \phi(\alpha,0)=\gamma_0,
\end{align*}
where $\dot\phi$ denotes the time derivative of $\phi$. The first term
on the right-hand side moves the string in the direction normal to the
string and the second term on the right-hand side is a Lagrange
multiplier term that acts in the tangential direction and enforces the
chosen parameterization. Numerical implementation consists of
discretizing the path into the images $\{\phi_i(t), i=0,1,\ldots, M\}$
and following a two-step method: first evolve according to
$\dot\phi_i=-\nabla V(\phi_i)^\perp$, and second reparameterize. Note
that the reparameterization is essential, since otherwise the points
will cluster around the minima.

The simplified String Method is based upon the observation that we can
absorb part of $(\nabla V(\phi))^\perp$ in the tangential component
and that doing so simplifies the algorithm and improves both stability
and accuracy~\cite{ERV2}. In this case, the evolution of the string is
best conceptualized as
\begin{align*}
  \dot \phi=-\nabla V(\phi)+\bar\lambda \hat\tau,
\end{align*}
where $\bar\lambda=(\nabla V(\gamma),\hat \tau)+\lambda$. The first
step is then to evolve according to the simplified dynamics
$\dot\phi_i=-\nabla V(\phi_i)$ and the second step is, as before, to
reparameterize. See~\cite{ERV2} for a full discussion.

We say that the string has converged when there is a point along the
string and away from the endpoints where the energy gradient vanishes
(to within some tolerance). In our application, we find only one
critical point between the local minima. We then read off the local
minimizers at the endpoints of the string and the saddle point as the
interior zero of the energy gradient.

\begin{remark}[Generalizations]
  We note that one does not need to know $x_*$ a priori. If one
  endpoint of the initial string lies in the basin of attraction of
  $x_*$ and is left free, the end of the string will evolve
  automatically toward the stable stationary point.

  We also note that the applicability of the String Method is not
  limited to gradient flows (although it has been mainly used in that
  context). One can consider
  \begin{align*}
    \dot x(t)=b(x),\qquad x(0)=x_0
  \end{align*}
  for general vector fields $b$ and use the String Method to find
  the orbit connecting $x_0$ and a zero $x_*$ of $b$, i.e., a
  path $\gamma$ connecting $x_0$ and $x_*$ such that
  \begin{align*}
    (b)^\perp (\gamma)=0.
  \end{align*}
\end{remark}

\subsection{Implementation and convergence}
\label{sec:impl-conv}

We implemented the String Method with a simple pseudo-spectral
algorithm and the FFT in MATLAB.  One string typically had a size of
$512$ images, although for $\phi\in\{0.025, 0.0125\}$ we had to take
reduce to $256$ to cut down the size of the string. Computations were
done in double precision.  The spatial mesh size was chosen to be
\begin{align*}
  N_{x}:=2\left\lceil \frac{\xi^{\frac{3}{2}}}{\phi^{\frac{3}{2}}}\right\rceil
\end{align*}
and the size of the time step was set to $N_{t}:=N_{x}^{2}\phi$.  When
computing the saddle point for $\phi=0.05$ and $\phi=0.0125$, we were
not able to reduce the tolerance to the value set as our stopping
criterion, so we let it run for as long as we could.  The tolerance
was reduced relatively quickly (approximately 439556.72 s) to
approximately the level that we achieved in the end. After that
results were slow to improve.

\begin{table}[h]
	\caption{Size of various parameters for $\xi=2.3$ and different values of $\phi$.}
    \begin{adjustbox}{width=\linewidth,center}	
      \begin{tabular}{lrrrrrrrrr}
		\toprule
		$\phi$ 						& 0.4000	& 0.2000	& 0.1000	& 0.0500	& 0.0250		& 0.0125		& 0.00625		& 0.003125	\\\midrule
		gridsize image 			& $28^{2}$	& $78^{2}$	& $222^{2}$	& $624^{2}$	& $1766^{2}$	& $4992^{2}$	& $14120^{2}$	& $39936^{2}$	\\
		image 					& 6\,KB		& 49\,KB	& 0.4\,MB	& 3.1\,MB	& 25\,MB		& 199\,MB		& 1.6\,GB		& 12.8\,GB	\\
		string	512 images		& 3\,MB		& 25\,MB	& 202\,MB	& 1.6\,GB	& 12.8\,GB		& 102\,GB		& 817\,GB		& 6533\,GB	\\
		CPU time for saddle	in s	& 15		& 138		& 970		& 7446510	& 1277578		& 				& 				& 	\\
		max RAM for saddle	in MB	& 664		& 746		& 1405		& 12731		& 45686			& 				& 				& 	\\
		CPU time for min in s		& 			& 			& 			& 			& 				& 9739			& 125993		& 	\\
		max RAM for min in MB		& 			& 			& 			& 			& 				& 5359			& 38419			& 	\\\bottomrule
	  \end{tabular}
    \end{adjustbox}
\end{table}

\subsection{Numerical results}
\label{ss:numer}

\renewcommand{\arraystretch}{1.5}
\begin{table}[ht]
  \caption{Numerical results for the local minimum for $\xi=2.3$ and
    different values of $\phi$.}\label{t:min}
  \begin{subtable}{\textwidth}
    \centering
    \caption{Raw data.}
    \begin{adjustbox}{width=\columnwidth,center}	
	  \begin{tabular}{lcccccccc}
		\toprule
		$\phi$ 																								& 0.4000 	& 0.2000	& 0.1000	& 0.0500	& 0.0250	& 0.0125	& 0.00625	 \\ \midrule
		$\lvert\{-1+2\phi^{\frac{1}{3}}\leq u_{m,\phi}\leq 1-2\phi^{\frac{1}{3}}\}\rvert$					& 0.0000 	& 0.0000	& 0.1675	& 0.2975	& 0.2459	& 0.1647	& 0.1029	 \\
		$\lvert\mathcal{E}_{\phi}^{\xi}(u_{m,\phi})-c_{m}\rvert$ 											& 3.8415 	& 2.1848	& 1.1512	& 0.5900	& 0.2986	& 0.1502	& 0.0753	 \\
		$\lvert\nu(u_{m,\phi})-\nu_{m}\rvert$ 																& 1.9640 	& 0.1711	& 0.1514	& 0.1773	& 0.1256	& 0.0834	& 0.0518	 \\
		$\lvert u_{m,\phi}-\Psi_{m}\rvert_{\mathbb{T}_{\phi L}}^{2}$ 										& 4.2930 	& 1.8807	& 0.9145	& 0.4526	& 0.2261	& 0.1132	& 0.0564	 \\
		$\mathrm{error}_{m,\phi}$ 																			& 0.0223 	& 0.0200	& 0.0053	& 0.0037	& 0.0031	& 0.0028	& 0.0027	 \\\bottomrule
	  \end{tabular}
    \end{adjustbox}
  \end{subtable}
  \newline
  \vspace*{0.5 cm}
  \newline
  \begin{subtable}{\textwidth}	
    \centering
    \caption{Log of ratios}
    \begin{adjustbox}{width=\columnwidth,center}	
	  \begin{tabular}{lccccccc}
		\toprule
		$\phi$ 			& 0.2000	& 0.1000	& 0.0500	& 0.0250	& 0.0125	& 0.00625	\\ \midrule
		$\log_{2}\left(\frac{\lvert\{-1+2(2\phi)^{\frac{1}{3}}\leq u_{m,2\phi}\leq 1-2(2\phi)^{\frac{1}{3}}\}\rvert}{\lvert\{-1+2\phi^{\frac{1}{3}}\leq u_{m,\phi}\leq 1-2\phi^{\frac{1}{3}}\}\rvert}\right)$
					& 			& 			& -0.8287	& 0.2748	& 0.5785	& 0.6782	\\
		 $\log_{2}\left(\frac{\lvert\mathcal{E}_{2\phi}^{\xi}(u_{m,\phi})-c_{m}\rvert}{\lvert\mathcal{E}_{\phi}^{\xi}(u_{m,\phi})-c_{m}\rvert}\right)$
					& 0.8142	& 0.9244	& 0.9644	& 0.9825	& 0.9915	& 0.9957	\\
		$\log_{2}\left(\frac{\lvert\nu(u_{m,2\phi})-\nu_{m}\rvert}{\lvert\nu(u_{m,\phi})-\nu_{m}\rvert}\right)$  	
					& 3.5209	& 0.1765	& -0.2278	& 0.4974	& 0.5899	& 0.6891	\\
		$\log_{2}\left(\frac{\lvert u_{m,2\phi}-\Psi_{m}\rvert_{\mathbb{T}_{2\phi L}}^{2}}{\lvert u_{m,\phi}-\Psi_{m}\rvert_{\mathbb{T}_{\phi L}}^{2}}\right)$ 		
					& 1.1907	& 1.0402	& 1.0147	& 1.0013	& 0.9985	& 1.0039	\\
		\bottomrule
	  \end{tabular}
    \end{adjustbox}
  \end{subtable}
\end{table}
Although the local minimizer can be approximated by steepest descent
without the String Method, we begin by summarizing the data recovered
for the minimizer; see Table \ref{t:min}.  Here the stopping criterion
was that the spectral norm of the difference of two consecutive
approximations of the minimizer is smaller than $5\cdot 10^{-3}$ and
the energy gradient at the minimum was at most $10^{-3}$.

Compared to \eqref{emin}, the numerical results seem to suggest the
better convergence rate $\phi^1$ for the energy and roughly
$\phi^{2/3}$ for the interfacial region and the droplet ``volume.''
The $L^2$ difference to the sharp interface profile also seems to
converge at rate $\phi^1$. These results are consistent with the
analysis but suggest that the analytical bounds to date are not
sharp. The level lines and a cross-section of the local minimizer are
shown in Figure \ref{fig:min}. See also Figure \ref{fig:minsad} for an
illustration when $\xi<\xi_d$.  \FloatBarrier

We now turn to the results of the String Method for the saddle
point. The numerical data is summarized in Table \ref{t:sad}. As for
the minimum, the results suggest better convergence in energy and
$L^2$ than reflected in \eqref{enps} and \eqref{ven2}. A good estimate
of the convergence rate for the volume and the interfacial volume
would require additional data, but the overall data support the
\emph{conjecture that the saddle point is indeed a volume-constrained
  minimum with ``volume'' approximately $\nu_s$}.

In Figure \ref{fig:closeup}, we plot a closeup of some of the contours
of the saddle point, which is a droplet-shaped function with the
expected properties. In addition to being connected and Steiner
symmetric about the axes, the superlevel sets for values near $1$
appear to be nearly spherical and convex.  We observe a unique minimum
point in the corner and (bearing in mind the periodicity) convexity of
the sublevel set $\{u\leq s\}$ for $s$ near this minimum
value. Motivated by these numerical observations, we examine convexity
in Section \ref{S:convex}.

\renewcommand{\arraystretch}{1.5}
\begin{table}[ht]
  \caption{Numerical results for the saddle point for $\xi=2.3$ and
    different values of $\phi$.}\label{t:sad}
  \begin{subtable}{\textwidth}
    \centering
    \caption{Raw data.}
	\begin{tabular}{lcccccc}
		\toprule
		$\phi$ 																								& 0.4000 	& 0.2000	& 0.1000	& 0.0500	& 0.0250		\\ \midrule
		$\lvert\{-1+2\phi^{\frac{1}{3}}\leq u_{s,\phi}\leq 1-2\phi^{\frac{1}{3}}\}\rvert$ 					& 0.0000 	& 0.0000	& 0.0387	& 0.0725	& 0.0612		\\
		$\lvert\mathcal{E}_{\phi}^{\xi}(u_{s,\phi})-c_{s}\rvert$ 											& 0.6101 	& 0.2041	& 0.0047	& 0.0182	& 0.0082		\\
		$\lvert\nu(u_{s,\phi})-\nu_{s}\rvert$ 																& 4.5279 	& 0.4664	& 0.0724	& 0.0475	& 0.0747		\\
		$\lvert u_{s,\phi}-\Psi_{s}\rvert_{\mathbb{T}_{\phi L}}^{2}$ 										& 5.6429 	& 2.9339	& 1.3946	& 0.7607	& 0.5849		\\
		$\mathrm{error}_{s,\phi}=\norm{\nabla E(u_{s,\phi})}^{2}$ 											& 0.0039 	& 0.0038	& 0.0008	& 0.0666	& 0.2922		\\
		\bottomrule
	\end{tabular}
  \end{subtable}
  \newline
  \vspace*{0.5 cm}
  \newline
  \begin{subtable}{\textwidth}	
    \centering
    \caption{Log of ratios}
	\begin{tabular}{lccccc}
		\toprule
		$\phi$																							& 0.2000	& 0.1000	 & 0.0500	& 0.0250		\\ \midrule
		$\lvert\{-1+2\phi^{\frac{1}{3}}\leq u_{s,\phi}\leq 1-2\phi^{\frac{1}{3}}\}\rvert$				& 			& 			 & -0.9055	& 0.2450		\\
		$\lvert\mathcal{E}_{\phi}^{\xi}(u_{s,\phi})-c_{s}\rvert$ 	 									& 1.5798	& 5.4405	 & -1.9547	& 1.1557		\\
		$\lvert\nu(u_{s,\phi})-\nu_{s}\rvert$ 	 														& 3.2792	& 2.6875	 & 0.6086	& -0.6531		\\
		$\lvert u_{s,\phi}-\Psi_{s}\rvert_{\mathbb{T}_{\phi L}}^{2}$ 									& 0.9436	& 1.0730	 & 0.8744	& 0.3792		\\
		\bottomrule
	\end{tabular}
  \end{subtable}
\end{table}
\FloatBarrier

\begin{figure}[t!]
    \centering
    \includegraphics[width=0.9\textwidth]{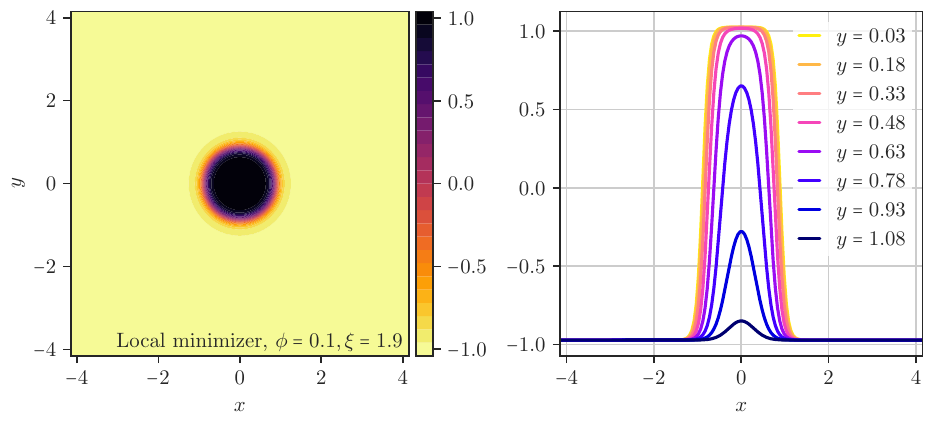}
    \caption{The local minimizer has the expected properties. Left:
      Level curves of the local minimum for $\phi = 0.1, \xi =
      1.9$. Right: Crossections at various $y$ as a function of
      $x$.}\label{fig:min}
\end{figure}

\begin{figure}[ht]
  \centering
  \includegraphics[width=0.9\textwidth]{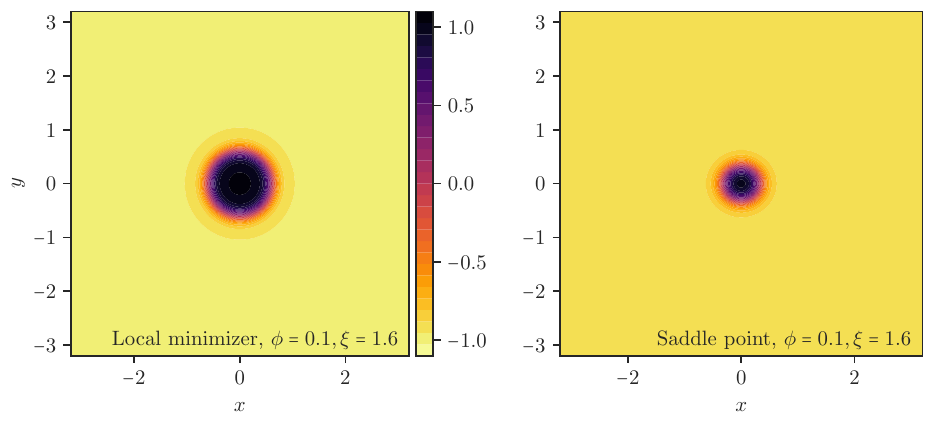}
  \caption{Local minimizer (left) and saddle point (right) for
    $\phi=0.1$ and $\xi=1.6$, i.e., in the regime that the constant
    state is not the global minimizer. The string to compute the
    saddle point is converged up to error
    $10^{-5}$.}\label{fig:minsad}
\end{figure}
\begin{figure}[ht]
  \centering
  \includegraphics[width=0.9\linewidth]{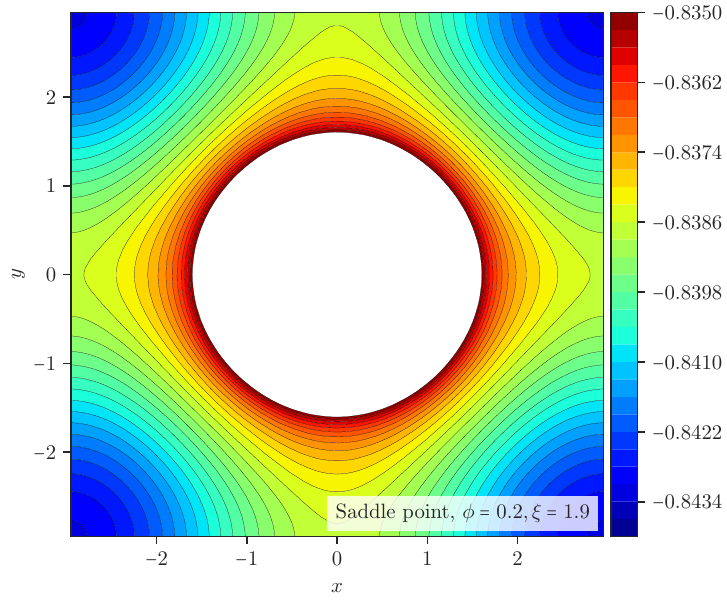}
  \caption{A close-up of the level lines of the saddle point for
    $\phi=0.2$ and $\xi=1.9$. The string is converged up to error
    $10^{-4}$. This figure motivates the analysis in Section
    \ref{S:convex}.}\label{fig:closeup}
\end{figure}

\FloatBarrier

\section{Convexity}\label{S:convex}
In the following sections we consider the volume-constrained minimizers; that is, we consider the minimizers of the energy \eqref{ephi}
over the set
\begin{align}\label{X-phi}
X_{\omega,\phi}:=\left\{u\in H^1\cap L^4(\tphil): \dashint_{\tphil} u \, dx=-1+\phi,\quad\nu(u)=\omega\right\},
\end{align}
where the ``volume'' $\nu(\cdot)$ is given by \eqref{nu}.
We are interested in geometric properties of the minimizers.
For simplicity of notation, we will write $u$ instead of $\uo$.

We would like to understand convexity of the superlevel sets of the constrained minimizers. Based on the numerics (see for instance Figure~\ref{fig:closeup}), it seems likely that the superlevel sets are convex not only near the maximum but for all values up to some critical value, at which point the convexity changes and the sublevel sets around the minimum become convex. Such a result is currently out of reach, but we will adapt a method of Caffarelli and Spruck to demonstrate convexity of superlevel sets under Hypothesis {\bf{(H2)}}, below. Before turning to this result, we summarize some basic facts from the previous analysis.

We begin by recalling two results from \cite{GWW} and \cite{GWW20}.
First, according to the mean condition, the definition of $\nu(\cdot)$,
and Lemma 2.9 from \cite{GWW}, one has the following uniform bounds.
\begin{lemma}\label{bound}
For any $\omega_1>0$ there exists $\phi_0>0$ such that for all $\phi\in(0,\phi_0)$, any volume-constrained minimizer $u$ for
$\omega\in[\omega_1,\xi^{3}/2]$ satisfies
\begin{align}
-1-\phi^{1/3}\leq\inf_{\tphil}u\leq -1+\phi\leq 1-2\phi^{1/3}\leq \sup_{\tphil}u\leq 1+\phi^{1/3}.\label{bd}
\end{align}
\end{lemma}
The upper bound $\phi_0$ from Lemma \ref{bound} is the only restriction we will place on $\phi>0$ for the rest of the article. As a consequence of the uniform bounds, any constrained minimizer is smooth and satisfies the Euler Lagrange equation
\begin{align}
-\phi\Delta u +\frac{1}{\phi}G'(u)+\lambda_{\phi}+\lambda_{\omega}\chi'(u)=0\qquad\hbox{in}\:\tphil,\label{eleq}
\end{align}
where $\lambda_{\phi}$, $\lambda_{\omega}\in\R$ are Lagrange parameters corresponding to the mean and volume constraints, respectively.
Next we recall the symmetry of the constrained minimizers.
We assume without loss of generality that the location of the unique maximum of $u$ is $x=0$.
By symmetry it is often sufficient to restrict to the first quadrant, i.e., to
$$\tphil^{+}:=\{x\in \tphil: x_{i} >0,\, i=1,\hdots,d\}.$$

\begin{theorem}[Theorem 1.1, \cite{GWW20}]\label{steinersym}
 Any volume-constrained minimizer $u$ is equal (up to a translation) to its Steiner symmetrization about the origin. In particular, its superlevel sets are simply connected, and
$u$ is strictly decreasing along all rays starting at the (unique) maximum of $u$. In fact for any ray $e$ starting in the maximum, the directional derivative is strictly negative:
\begin{align}
\partial_{e}u< 0\qquad\text{on }\tphil^+.\label{eneg}
\end{align}
\end{theorem}
\begin{remark} The first part of the theorem is precisely Theorem 1.1 from \cite{GWW20}, and~\eqref{eneg} and the nonvanishing of the gradient follow from the Strong Maximum Principle, as shown in the proof of Theorem 1.1 in \cite{GWW20}.
\end{remark}
\subsection{Convexity method of Caffarelli and Spruck}\label{ss:convex}
We will now use the Euler Lagrange equation \eqref{eleq} and a method that Caffarelli and Spruck introduced in~\cite[Theorem 2.1]{CS}
to establish convexity of the superlevel sets of volume-constrained minimizers $u=\uo$ under an assumption on the nonlinearity, specified in Hypothesis {\bf{(H2)}}, below.

\begin{remark}\label{nondeg}
By symmetry of $u$ we have for $i\ne j$ that
\begin{align*}
\partial_{i}u(0,\hdots,0,x_j,0,\hdots,0)=0.
\end{align*}
Hence the Hessian matrix $D^2u(0)$ is diagonal, and we have
\begin{align*}
\left(D^2u(0)\right)_{ij}=\lambda_i\delta_{ij}
\end{align*}
where $(\delta_{ij})_{ij}$ denotes the identity matrix in $\R^d$ and $\lambda_i\in\R$ for $i\in\{1,\hdots,d\}$. Because $x=0$ is a maximum point of $u$, we have $\lambda_i\leq 0$ for $i\in\{1,\hdots,d\}$. More is true:
\end{remark}
\begin{lemma}\label{convex}
The origin is a non-degenerate critical point of any volume-constrained minimizer $u$,  and $\partial^2_{e}u(0)<0$ for any ray $e$ starting in the origin.
\end{lemma}
\begin{proof}
Let $e$ be any ray starting in the origin. We choose $i\in\{1,\hdots,d\}$ such that
$e$ points into the halfspace $\tphil^{i}:=\{x\in \tphil: x_{i} >0\}$.
According to Theorem \ref{steinersym}, we have
$\partial_{e}u< 0$ in $\tphil^{i}$. Note that by \eqref{eleq} the
function $\partial_{e}u$ satisfies the equation
\begin{align*}
-\phi\Delta(\partial_{e}u)+\left(\frac{1}{\phi}G''(u)+\lambda_{\omega}\:
\chi''(u)\right)\partial_{e}u=0\qquad \text{in $\tphil^{i}$. }
\end{align*}
Also for $\phi$ sufficiently small, there holds
\begin{align}
 \sup_{\{u>1-\phi^{\frac{1}{3}}\}} \frac{1}{\phi}G''(u)>0.\label{cpos}
\end{align}
For fixed $0<\phi\ll 1$, we choose a radius $\rho>0$ small enough so that the ball $B_\rho(0)$ satisfies
\[B_{\rho}(0)\Subset \{u>1-\phi^{\frac{1}{3}}\}.\]
Then the definition of $\chi$ implies
\begin{align*}
\frac{1}{\phi}G''(u)+\lambda_{\omega}\:\chi''(u)=\frac{1}{\phi}G''(u)\,\overset{\eqref{cpos}}>\,0\qquad\text{for all }x\in B_{\rho}(0).
\end{align*}
Because the function $\partial_{e}u$ takes its maximum on $\overline{\tphil^{i}\cap B_{\rho}(0)}$ in $x=0$,
Hopf's maximum principle hence yields
$\lambda=\partial^2_{e}u(0)<0$. Consequently $x=0$ is a non-degenerate maximum point.
\end{proof}
As a consequence of Lemma \ref{convex},  the superlevel sets $S_c:=\{u>c\}$
are strictly convex for all $c$ sufficiently close to $u_{max}:=\sup_{\tphil} u$, which
according to \eqref{bd} is close to $1$. Thus there exists $c^{*}<u_{max}$ such that $S_{c}$
is convex for all $c^{*}\leq c<u_{max}$.
\medskip

We would now like to investigate convexity of superlevel sets of $u$ further away from $u_{max}$.
Suppose that  $c_{*}>\inf_{\tphil}u$ is some value such that
\begin{quote}\textbf{Hypothesis (H1)}:
\begin{align*}
S_{c_{*}}=\{u>c_{*}\}\text{ is convex and precompact and } \partial S_{c_{*}}\text{ is smooth.}
\end{align*}
\end{quote}
We assume that $c^{*}>c_{*}$ since otherwise there is nothing to prove.
We will show that all sublevel sets $S_c$ for $c_{*}< c<c^{*}$ are convex.
\medskip

In the remainder of this section we will write the Euler-Lagrange equation of $u$ in the form
$\Delta u=f(u)$, where
\begin{align}
f(u):=\frac{1}{\phi^2}G'(u)+\frac{\lambda_{\phi}}{\phi}+\frac{\lambda_{\omega}}{\phi}\chi'(u).\label{rhs}
\end{align}
Note that because of the smoothness of $G$ and $\chi$, $f$ is also smooth. We will work under the assumption:
\begin{quote}\textbf{Hypothesis (H2)}:
\begin{align*}
f(c_{*})\geq 0 \quad\text{and}\quad f'(c)> 0\quad\hbox{for all}\quad c_{*}<c<u_{max}.
\end{align*}
\end{quote}
Our main theorem is:
\begin{theorem}\label{convex5new}
Let $u$ solve $\Delta u=f(u)$ in $\tphil$, where $f$ is given by \eqref{rhs}.
Assume $u$ takes its unique maximum point $u_{max}$ in $x=0$ and $\partial_{e}u< 0$ along all rays $e$ starting at the origin.
Let $c_{*}<u_{max}$ be a level  such that Hypotheses{ \bf (H1)} and {\bf (H2)} hold.
Then $S_{c}$ is convex
for all $c_{*}< c<u_{max}$.
\end{theorem}
\medskip

In light of Theorem~\ref{steinersym} and translating if necessary, the assumption of the unique maximum in zero and directional derivative on rays is without loss of generality for the desired application.

Following the strategy of \cite{CS}, we will work with the following definitions.
\begin{definition}
We will call three  points $x,y,z\in S_{c_{*}}$ a \uline{triple} if $z=tx+(1-t)y$ for some $t\in[0,1]$, and we will write $z\in [x,y]$ for short and speak of ``the triple'' $(x,y,z)$.  We will call a triple \uline{admissible} if it is an element of
\begin{align}
\label{quasicon0} \A := \{ (x,y,z) \colon   x,y \in S_{c_{*}},\; x,y,z\hbox{ are a triple},\, \text{and } \: u(x)>u(z)\}.
\end{align}
\end{definition}

Note that $\A$ is not empty since $u$ is strictly decreasing along all rays emanating from the origin.
Now consider the function
\begin{align*}
\F(x,y,z):=u(y)-u(z),\qquad\hbox{for}\quad (x,y,z)\in\A.
\end{align*}
Because of the equivalence in Lemma~\ref{convex1} below, we will be interested in
\begin{align*}
 \sup_{(x,y,z)\in\A}\F(x,y,z).
\end{align*}
Let $(x_k,y_k,z_k)\in\A$ be a maximizing sequence for $\F$.
By compactness of $\overline{S_{c_{*}}}$, there exist limit
points $x_0,y_0,z_0\in \overline{S_{c_{*}}}$. By passing to suitable subsequences,
we may assume that $x_k\to x_0$, $y_k\to y_0$ and $z_k\to z_0$.
Moreover
\begin{align}
z_0=t_0 x_0+(1-t_0)y_0\qquad \text{for some $t_0\in[0,1]$}\label{tzero}
\end{align}
and $u(x_0)\geq u(z_0)$. A priori it may be that
$x_0=z_0$ or $z_0=y_0$ or even $x_0=z_0=y_0$, however see Remark~\ref{contra}.
\medskip
In \cite{CS} it is observed that $\F$ can be used to capture convexity:
\begin{lemma}\label{convex1}
The convexity of $S_c$ for all $c_{*}\leq c\leq u_{max}$ is equivalent to
\begin{align}
\label{quasicon1} \sup_{(x,y,z)\in\A}\F(x,y,z)\leq 0.
\end{align}
\end{lemma}
\begin{proof}
 \uline{Step 1}. We begin by assuming that the sets $S_c$ are convex for all $c_{*}<c<u_{max}$.
 Let $(x,y,z)\in\A$. If $u(y)>u(x)$, then convexity of the
superlevel sets gives $u(z)\geq u(x)$ for all $z\in [x,y]$ and $(x,y,z)$ is not admissible.
Hence any admissible triple satisfies $u(x)\geq u(y)$. However then
the convexity of the superlevel sets gives $u(z)\geq u(y)$, so that $\F(x,y,z)\leq 0$, and taking the supremum over
$\A$ preserves this inequality.
\newline
\uline{Step 2}. We now assume $\sup_{(x,y,z)\in\A}\F(x,y,z)\leq 0$ and will show that $S_{c}$ is convex for all $c_{*}<c<u_{max}$ by showing that for any
$x,y\in S_{c}$ and any $z\in[x,y]$, there holds $z\in S_{c}$.
On the one hand, if $(x,y,z)$ is admissible, then our assumption on $\F$ yields $u(z)\geq u(y)>c$ (and $z\in S_c$). On the other hand if $(x,y,z)$
is not admissible, then by definition~\ref{quasicon0} there holds
$u(z)\geq u(x)>c$ (and $z\in S_c$).
\end{proof}
In order to prove the convexity of $S_c$ for all $c_{*}\leq c\leq u_{max}$, we will use Lemma \ref{convex1} and argue by
contradiction. Thus we assume that \eqref{quasicon1} does not hold and we will discuss the extremal situation, by which we mean the following:
\begin{definition}\label{def:max}
We will call $(x_0,y_0,z_0)$ a \uline{positive extremal} if there exists a maximizing sequence $(x_k,y_k,z_k)\in\A$ with $x_k\to x_0$, $y_k\to y_0$ and $z_k\to z_0$ such that
\begin{itemize}
\item[($P_1$)]\quad $x_0,y_0,z_0\in\overline{S_{c_{*}}}$ and $z_0=t_0 x_0+(1-t_0)y_0$ for some $t_0\in[0,1]$;
\item[($P_2$)]\quad $u(x_0)\geq u(z_0)$;
\item[($P_3$)]\quad there holds $0<\sup_{(x,y,z)\in\A}\F(x,y,z)=\F(x_0,y_0,z_0)$, so that
\begin{align}\label{convexass2}
u(y_0)>u(z_0).
\end{align}
\end{itemize}
\end{definition}
Our plan in order to show  convexity is as follows.
\begin{itemize}
\item Using maximizing sequences $(x_k,y_k,z_k)$,
we  derive necessary conditions for any positive extremal $(x_0,y_0,z_0)$.
The proofs rely heavily on Lemma \ref{convex}, however they do not explicitly use the equation for $u$. In particular {\bf{(H2)}} is not needed.

\item We then assume for a contradiction that $\sup_{(x,y,z)\in\A } \F(x,y,z)>0$. It is only here that Hypothesis {\bf{(H2)}} is needed.
The necessary conditions for the limit $(x_0,y_0,z_0)$ are used to define admissible sequences $(x_k,y_k,z_k)$ that lead to a contradiction.
\end{itemize}

\begin{remark}\label{contra}
Clearly for any positive extremal there holds $y_0\ne z_0$ (i.e., $t_0\neq 0$), since this would imply $\F(x_0,y_0,z_0)=0$, in violation of ($P_3$). Also
$x_0\ne y_0$, since this would yield $x_0=y_0=z_0$ (by ($P_1$)), again violating  ($P_3$). On the other hand, the case $x_0=z_0$ cannot be excluded.
\end{remark}

\begin{lemma}\label{y0eq0}
If  the triple $(x_0,y_0,z_0)$ is a positive extremal, then $y_0\ne 0$.
\end{lemma}
\begin{proof}
For a contradiction we consider any triple $(x_0,y_0,z_0)$ with $y_0=0$, i.e.,
$y_0$ is the unique point
where $u$ takes its maximum. We will show that $(x_0,y_0,z_0)$ is not the limit of a sequence of admissible
triples $(x_k,y_k,z_k)\in\A$ as $k\to\infty$. Thus assume that there exists a sequence $(x_k,y_k,z_k)\in\A$
such that $\lim_{k\to\infty}(x_k,y_k,z_k)=(x_0,y_0,z_0)$.
We distinguish two cases.

\medskip

\noindent 1. If $y_0=0$ and $x_0\ne z_0$,
then by ($P_1$), we have $z_0=t x_0$ for some $t\in (0,1]$. Since by Lemma \ref{convex}),
 $u$ is strictly decreasing along any ray starting in the origin, there holds $u(z_0)>u(x_0)$, which contradicts ($P_2$).

\medskip

\noindent
2. Now suppose $y_0=0$ and $x_0= z_0$.
Because $(x_k,y_k,z_k)$ is admissible, we have $u(x_k)>u(z_k)$. From the collinearity of $(x_0,y_0,z_0)$ and $y_0=0$, we deduce that
\begin{align*}
\frac{x_k-z_k}{\vert x_k-z_k\vert}=\frac{x_k-y_k}{\vert x_k-y_k\vert}\to\frac{x_0}{\vert x_0\vert}=: e_0\quad\text{as $k\to\infty$.}
\end{align*}
Hence
\begin{align*}
0<\frac{u(x_k)-u(z_k)}{\vert x_k-z_k\vert}
=
\int\limits_{0}^{1}\frac{x_k-z_k}{\vert x_k-z_k\vert}\cdot Du(s x_k+(1-s)z_k)\:ds
\to
e_0\cdot Du(x_0)=\partial_{e_{0}}u(x_0),
\end{align*}
and we obtain in the limit $k\to\infty$ that
$
\partial_{e_{0}}u(x_0)\geq 0,
$
which contradicts the fact that the directional derivative is strictly negative along rays from the origin (cf.\ Theorem \ref{steinersym}).

\end{proof}
By definition, any positive extremal lies in $\overline{S_{c_{*}}}$; we now show that, in fact, any positive extremal lies in the interior.
\begin{lemma}\label{convex21}
For any positive extremal $(x_0,y_0,z_0)$, there holds
\begin{align*}
\text{(i) }\,y_0\in S_{c_{*}}\quad\text{ and }\quad
\text{(ii) } \,z_0\in S_{c_{*}}.
\end{align*}
\end{lemma}
\begin{proof}
Since necessarily $y_0, z_0\in\overline{S_{c_{*}}}$, the issue is to show $y_0, z_0\notin \partial S_{c_{*}}$. We begin by considering $y_0$. Since $z_0\in\overline{S_{c_{*}}}$, we have $u(z_0)\geq c_{*}$. Combining this with \eqref{convexass2} implies (i).

We now turn to $z_0$. Let us assume for a contradiction that $z_0\in\partial S_{c_{*}}$,
i.e., $u(z_0)=c_{*}$. By (i) we have $y_0\in S_{c_{*}}$. By ($P_1$) the points $x_0$, $y_0$
and $z_0$ are collinear. Since $S_{c_{*}}$ is convex, the ray that starts in $y_0$ and passes
through $z_0$ and $x_0$ intersects $\partial S_{c_{*}}$ in exactly one point. By ($P_1$) and
($P_2$) the point $x_0$ must be in $\partial S_{c_{*}}$ and thus $x_0=z_0$. Now let
$(x_k,y_k,z_k)\in\A$ be a maximizing sequence converging to $(x_0,y_0,z_0)$ as
in Definition~\ref{def:max}. For sufficiently large $k$ the following holds:
Since $x_k,z_k\to x_0=z_0\in\partial S_{c_{*}}$, for $k$ large there exists a ray that
starts in $y_k$, passes through
$z_k$ and $x_k$ and intersects $\partial S_{c_{*}}$ in exactly one point.
This point is denoted by $\tilde{x}_k$. Note that $x_k\in[\tilde{x}_{k},z_k]$ and $\tilde{x}_k\to x_0$ as $k\to\infty$. Since {\bf{(H1)}} holds and since $\partial_{e}u< 0$ on rays from the origin, for $k$ sufficiently large there holds $\partial_{\nu}u(\tilde{x}_k)<0$, where $\nu$ denotes the outer unit normal vector in $\tilde{x}_k\in \partial S_{c_{*}}$. This implies that $u$ is strictly decreasing along $[z_k,\tilde{x}_k]$ and hence $u(z_k)>u(x_k)$, contradicting the admissibility of $(x_k,y_k,z_k)$. As a consequence $z_0\notin\partial S_{c_{*}}$.
\end{proof}
\subsection{Optimality conditions for extremals}
We assume $(x_0,y_0,z_0)$ is a positive extremal, so that in particular \eqref{convexass2} holds.
Note that
\begin{align*}
z_0=t_0\,x_0+(1-t_0)\,y_0
\end{align*}
for some $t_0\in (0,1]$.
Our first optimality condition is:
\begin{lemma}\label{optcond2}
For any positive extremal with $0<t_0\leq 1$ we have $(y_0-x_0)\cdot\nabla u(z_0)=0$.
\end{lemma}

\begin{proof}
We remark that according to Definition~\ref{def:max}, there exists an admissible sequence $(x_k,y_k,z_k)$ converging to $(x_0,y_0,z_0)$, i.e., a sequence such that
\begin{align}
  u(x_k)-u(z_k)>0.\label{posxz}
\end{align}
Suppose $t_0=1$. Expanding \eqref{posxz} around $x_0=z_0$ using
\begin{align*}
  z_k=t_k x_k+(1-t_k)y_k=x_0+(1-t_k)(y_0-x_0+(y_k-y_0))+t_k(x_k-x_0)
\end{align*}
gives:
\begin{align*}
  0< -(1-t_k)(y_0-x_0)\cdot \nabla u(x_0)+ (1-t_k)\Big((x_k-x_0)\cdot \nabla u(x_0)-(y_k-y_0)\cdot\nabla u(x_0)   \Big) +O((1-t_k)^2),
\end{align*}
so that $(y_0-x_0)\cdot \nabla u(x_0)\leq 0$. However if the strict inequality holds, we can choose $\epsilon>0$ and  $z_\epsilon:=z_0+\epsilon (y_0-z_0)$ so that $(x_0,y_0,z_\epsilon)$ is admissible and contradicts the optimality of $(x_0,y_0,z_0)$. Therefore $(y_0-x_0)\cdot \nabla u(x_0)=0$.

In case $t_0\in (0,1)$ the proof is similar but simpler, since we can choose perturbations of $z_0$ in either direction along $[x_0,y_0]$.
\end{proof}

For the remainder of our arguments, we will use a sequence $(t_k)_k$ with $t_k\in [0,1]$, $\lim_{k\to\infty}t_k=t_0$ and the perturbations
\begin{align}
\label{xk} x_k&:=x_0+(t_0-t_k)\Phi_{x}^{1}+(t_0-t_k)^2\Phi_{x}^{2}\\
\label{yk} y_k&:=y_0+(t_0-t_k)\Phi_{y}^{1}+(t_0-t_k)^2\Phi_{y}^{2},
\end{align}
where $\Phi_{x}^{1}, \Phi_{x}^{2}, \Phi_{y}^{1}$ and $\Phi_{y}^{2}$ are given vectors.
At this point, our proofs diverge from those of \cite{CS}. Our goal is to give a unified and elementary treatment of all the necessary results in terms only of these perturbations and the associated optimality conditions.

For $\vert \Phi_{x,y}^{1,2}\vert<r$ and $r>0$ sufficiently small, Lemma \ref{convex21} implies that $x_k$, $y_k$ and hence $z_k$ are in $S_{c_{*}}$.
In order to obtain an admissible triple $(x_k,y_k,z_k)$, we must on the one hand have
\begin{align*}
z_k=t_k\,x_k+(1-t_k)\,y_k
\end{align*}
which is equivalent to
\begin{align}
\label{zk} z_k
&=
z_0+(t_0-t_k)\left[y_0-x_0+t_0\,\Phi_{x}^{1}+(1-t_0)\,\Phi_{y}^{1}\right]\\
\nonumber&+
(t_0-t_k)^2\left[\Phi_{y}^1-\Phi_{x}^1+t_0\,\Phi_{x}^2+(1-t_0)\Phi_{y}^2\right]
+
(t_0-t_k)^3\left[\Phi_{y}^2-\Phi_{x}^2\right],
\end{align}
and on the other hand satisfy $u(x_k)-u(z_k)>0$. An expansion with respect to $t_k$ at $t=t_0$ yields necessary conditions for the perturbations $\Phi_{x}^{1}, \Phi_{y}^{1}, \Phi_{x}^{2}$ and $\Phi_{y}^{2}$: An admissible perturbation must satisfy:
\begin{align}\label{optcond0}
0&\overset{!}{<}u(x_k)-u(z_k)=u(x_0)-u(z_0)\\
\nonumber&
+
(t_k-t_0)
\left\{t_0\,\Phi_{x}^{1}\cdot\nabla u(z_0)+(1-t_0)\,\Phi_{y}^{1}\cdot\nabla u(z_0)-\Phi_{x}^{1}\cdot\nabla u(x_0)\right\}\\
\nonumber&
+ (t_0-t_k)^2\biggl\{
-
\frac{t_0^2}{2}\left(\Phi_{x}^{1}-\Phi_{y}^{1}\right)\cdot D^2u(z_0)\cdot\left(\Phi_{x}^{1}-\Phi_{y}^{1}\right)
-
t_0\left(\Phi_{x}^{1}-\Phi_{y}^{1}\right)\cdot D^2u(z_0)\cdot\left(y_0-x_0\right)\\
\nonumber&
-
t_0\left(\Phi_{x}^{1}-\Phi_{y}^{1}\right)\cdot D^2u(z_0)\cdot\Phi_{y}^{1}
-
\frac{1}{2}\left(y_0-x_0\right)\cdot D^2u(z_0)\cdot\left(y_0-x_0\right)
-
\left(y_0-x_0\right)\cdot D^2u(z_0)\cdot\Phi_{y}^{1}\\
\nonumber&
-
\frac{1}{2}\Phi_{y}^{1}\cdot D^2u(z_0)\cdot\Phi_{y}^{1}
+
\left(\Phi_{x}^{1}-\Phi_{y}^{1}\right)\cdot\nabla u(z_0)
-
t_0\left(\Phi_{x}^{2}-\Phi_{y}^{2}\right)\cdot\nabla u(z_0)
-
\Phi_{y}^{2}\cdot\nabla u(z_0)\\
\nonumber&
+
\frac{1}{2}\Phi_{x}^{1}\cdot D^2u(x_0)\cdot\Phi_{x}^{1}
+
\Phi_{x}^2\cdot\nabla u(x_0)
\biggr\}
+
O(\vert t_0-t_k)\vert^3).
\end{align}
In view of \eqref{convexass2}, all admissible sequences satisfy
\begin{align}\label{maximality}
0&\geq u(y_k)-u(z_k)- \left(u(y_0)-u(z_0)\right)\\
\nonumber&=
 (t_k-t_0)\left\{t_0\,\Phi_{x}^{1}\cdot\nabla u(z_0)+(1-t_0)\,\Phi_{y}^{1}\cdot\nabla u(z_0)-\Phi_{y}^{1}\cdot\nabla u(y_0)\right\}\\
\nonumber&
+
(t_0-t_k)^2\biggl\{
-
\frac{t_0^2}{2}\left(\Phi_{x}^{1}-\Phi_{y}^{1}\right)\cdot D^2u(z_0)\cdot\left(\Phi_{x}^{1}-\Phi_{y}^{1}\right)
-
t_0\left(\Phi_{x}^{1}-\Phi_{y}^{1}\right)\cdot D^2u(z_0)\cdot\left(y_0-x_0\right)\\
\nonumber&
-
t_0\left(\Phi_{x}^{1}-\Phi_{y}^{1}\right)\cdot D^2u(z_0)\cdot\Phi_{y}^{1}
-
\frac{1}{2}\left(y_0-x_0\right)\cdot D^2u(z_0)\cdot\left(y_0-x_0\right)
-
\left(y_0-x_0\right)\cdot D^2u(z_0)\cdot\Phi_{y}^{1}\\
\nonumber&
-
\frac{1}{2}\Phi_{y}^{1}\cdot D^2u(z_0)\cdot\Phi_{y}^{1}
+
\left(\Phi_{x}^{1}-\Phi_{y}^{1}\right)\cdot \nabla u(z_0)
-
t_0\left(\Phi_{x}^{2}-\Phi_{y}^{2}\right)\cdot\nabla u(z_0)
-
\Phi_{y}^{2}\cdot\nabla u(z_0)
\\
\nonumber&
+
\frac{1}{2}\Phi_{y}^{1}\cdot D^2u(y_0)\cdot\Phi_{y}^{1}
+
\Phi_{y}^2\cdot\nabla u(y_0)
\biggr\}
+
O(\vert t_0-t_k\vert^3).
\end{align}
In both \ref{optcond0} and \ref{maximality}, the term $(y_0-x_0)\cdot\nabla u(z_0)$ vanishes because of Lemma \ref{optcond2}.

The core of most of the following arguments is: All perturbations that fulfill inequality \eqref{optcond0} must necessarily also fulfill inequality \eqref{maximality}. As a first consequence, we obtain $u(x_0)=u(z_0)$ (so that the first term in \eqref{optcond0} vanishes).
\begin{lemma}\label{optcond1}
For any positive extremal $(x_0,y_0,z_0)$ there holds $u(x_0) = u(z_0)$.
\end{lemma}
\begin{proof}
The statement is trivial when $t_0=1$ ($x_0=z_0$), so we assume $t_0\in(0,1)$. From \eqref{quasicon0} it follows that $u(x_0)\geq u(z_0)$. We assume for a contradiction that $u(x_0)>u(z_0)$. Thus there exists an index $k_0$ depending on $r>0$ such that for all
$\vert\Phi_{x,y}^{1,2}\vert<r$, the perturbations given by \eqref{xk} and \eqref{yk} satisfy
$u(x_k)-u(z_k)>0$ for all
$k\geq k_0$, i.e., \eqref{optcond0} is satisfied.
\medskip

\noindent
Then from \eqref{maximality} (choosing $k_0$ even larger if necessary), we deduce that all such perturbations fulfill
\begin{align*}
(t_k-t_0)\left\{(y_0-x_0)\cdot\nabla u(z_0)+t_0\,\Phi_{x}^{1}\cdot\nabla u(z_0)+
(1-t_0)\,\Phi_{y}^{1}\cdot\nabla u(z_0)-\Phi_{y}^{1}\cdot\nabla u(y_0)\right\}\leq 0.
 \end{align*}
Since we may choose the sign of $t_k-t_0$, the term in the curly brackets must vanish:
\begin{align}
\Phi_{y}^{1}\cdot \nabla u(y_0)
=
\left[(1-t_0)\Phi_{y}^{1}+t_0\Phi_{x}^{1}+(y_0-z_0)\right]\cdot\nabla u(z_0).\label{nextphi}
\end{align}
By choosing $\Phi_{x}^{1}=\Phi_{y}^{1}=0$, we obtain the necessary condition on the triple $\nabla u(z_0)\cdot(y_0-z_0)=0$, and \eqref{nextphi} simplifies to
\begin{align}
\Phi_{y}^{1}\cdot \nabla u(y_0)
=
\left[(1-t_0)\Phi_{y}^{1}+t_0\Phi_{x}^{1}\right]\cdot\nabla u(z_0)\label{nextphi2}
\end{align}
for all $\vert\Phi_{x,y}^{1,2}\vert<r$ and $k\geq k_0$.
\medskip

\noindent
Now choosing $\Phi_{y}^{1}=0$ in \eqref{nextphi2} gives
\begin{align*}
\Phi_{x}^{1}\cdot\nabla u(z_0)=0\qquad\text{for all\;}\vert\Phi_{x}^{1}\vert<r,
\end{align*}
which implies $\nabla u(z_0)=0$. Since $x=0$ is the only critical point of $u$ in
$S_{c_{*}}$ and $u(z_0)< u(y_0)$, this is a contradiction and concludes the proof that $u(x_0)=u(z_0)$.
\end{proof}
For the rest of the proofs, we will separately consider the case $0<t_0<1$ and the case $t_0=1$. (Recall from Remark~\ref{contra} that $t_0\neq 0$.) For the convenience of the reader, we collect two more optimality conditions here whose proofs for $t_0<1$ and $t_0=1$ are given separately below in Subsections~\ref{sec1} and \ref{sec3}.
\begin{lemma}\label{optcond22}
For any positive extremal $(x_0,y_0,z_0)$ we have
\begin{align}\label{optcond3}
\frac{\nabla u(z_0)}{\vert \nabla u(z_0)\vert}=
\frac{\nabla u(y_0)}{\vert \nabla u(y_0)\vert}=
\frac{\nabla u(x_0)}{\vert \nabla u(x_0)\vert}=:\xi_{0}.
\end{align}
Moreover in the case $t_0=1$, there holds $\vert\nabla u(y_0)\vert\ne\vert\nabla u(z_0)\vert$.
\end{lemma}
\begin{lemma}\label{optcond007}
Any positive extremal $(x_0,y_0,z_0)$ satisfies
\begin{align}\label{optcond6}
(y_0-x_0)\cdot D^2u(z_0)\cdot(y_0-x_0)=0.
\end{align}
\end{lemma}
\subsubsection{The case $0<t_0<1$}\label{sec1}
We begin by proving Lemmas \ref{optcond22} and \ref{optcond007} in the case $t_0<1$ and then collect two additional minor lemmas that we will use for the contradiction argument in this case.
It will be convenient in the remainder to denote
\begin{align}\label{ab}
a:=\frac{\vert\nabla u(x_0)\vert}{\vert\nabla u(z_0)\vert},\quad
b:=\frac{\vert\nabla u(y_0)\vert}{\vert\nabla u(z_0)\vert},\quad
X:=\Phi_{x}^{1}\cdot\xi_0,\quad\hbox{and}\quad Y:=\Phi_{y}^{1}\cdot\xi_0.
\end{align}
\begin{proof}[Proof of Lemma~\ref{optcond22} for $0<t_0<1$]
$ $\newline
\noindent
Taking into account Lemmas \ref{optcond2} and \ref{optcond1}, the first order conditions for $t_k>t_0$ can be summarized as:
\begin{align}\label{nec1}
&\Phi_{x}^{1}\cdot\nabla u(x_0)
\leq
\left(t_0\,\Phi_{x}^{1}+(1-t_0)\,\Phi_{y}^{1}\right)\cdot\nabla u(z_0)\\
\nonumber&\quad\Longrightarrow\quad
\Phi_{y}^{1}\cdot\nabla u(y_0)\geq
\nabla u(z_0)\cdot\left(t_0\,\Phi_{x}^{1}+(1-t_0)\,\Phi_{y}^{1}\right).
\end{align}
Choosing $\Phi_{x}^1=0$ and
$\Phi_{y}^{1}\cdot\nabla u(y_0)=0$, we obtain
\begin{align*}
0\geq \nabla u(z_0)\cdot\Phi_{y}^{1}\quad&\Longrightarrow\quad 0\leq\nabla u(z_0)\cdot\Phi_{y}^{1},
\end{align*}
so that  $\nabla u(z_0)\cdot\Phi_{y}^{1}=0$.
Thus $\nabla u(y_0)$ points in the same direction as $\nabla u(z_0)$.
Similarly the choice $\Phi_{y}^1=0$ and  $\Phi_{x}^{1}\cdot\nabla u(x_0)=0$ implies  that
$\nabla u(x_0)$ points in the same direction as $\nabla u(z_0)$.
\end{proof}
\begin{proof}[Proof of Lemma~\ref{optcond007} for $0<t_0<1$]
$ $\newline
\noindent
1. We choose $\Phi_{x}^{1}=\Phi_{y}^1=0$ so that the first-order terms in  \eqref{optcond0} and \eqref{maximality} vanish.
In this case inequality \eqref{optcond0} yields the admissibility condition
\begin{align}\label{second1}
0&\overset{!}\leq-\frac{1}{2}(y_0-x_0)\cdot D^2u(z_0)\cdot(y_0-x_0)
-
(1-t_0)\Phi_{y}^{2}\cdot\nabla u(z_0)\\
\nonumber&
-
t_0\Phi_{x}^{2}\cdot\nabla u(z_0)
+
\Phi_{x}^2\cdot\nabla u(x_0).
\end{align}
For such perturbations, \eqref{maximality} yields the necessary condition
\begin{align}\label{second2}
0&\geq
-
\frac{1}{2}\left(y_0-x_0\right)\cdot D^2u(z_0)\cdot\left(y_0-x_0\right)
-
(1-t_0)\Phi_{y}^{2}\cdot\nabla u(z_0)\\
\nonumber&-
t_0\Phi_{x}^{2}\cdot\nabla u(z_0)
+
\Phi_{y}^2\cdot\nabla u(y_0).
\end{align}
\medskip

\noindent

2. We begin by ruling out $(y_0-x_0)\cdot D^2u(z_0)\cdot(y_0-x_0)>0$. Using Lemma \ref{optcond22} and \eqref{ab}, \eqref{second1} and  \eqref{second2}can be expressed in the form:
\begin{align*}
\text{if}\qquad 0&\leq-\frac{1}{2\vert\nabla u(z_0)\vert}(y_0-x_0)\cdot D^2u(z_0)\cdot(y_0-x_0)
-
(1-t_0)\Phi_{y}^{2}\cdot\xi_{0}\\
\nonumber&
-
t_0\Phi_{x}^{2}\cdot\xi_{0}
+
a\Phi_{x}^2\cdot\xi_{0},
\end{align*}
then
\begin{align*}
0&\geq-\frac{1}{2\vert\nabla u(z_0)\vert}(y_0-x_0)\cdot D^2u(z_0)\cdot(y_0-x_0)
-
(1-t_0)\Phi_{y}^{2}\cdot\xi_{0}\\
\nonumber&
-
t_0\Phi_{x}^{2}\cdot\xi_{0}
+
b\Phi_{y}^2\cdot\xi_{0}.
\end{align*}
We obtain a contradiction by choosing $\Phi_{x}^{2}$ and $\Phi_{y}^{2}$ such that
\begin{align*}
a\,\Phi_{x}^{2}\cdot\xi_{0}<b\,\Phi_{y}^{2}\cdot\xi_{0}.
\end{align*}

\medskip

\indent
3. On the other hand
 $(y_0-x_0)\cdot D^2u(z_0)\cdot(y_0-x_0)<0$ is ruled out since then all sufficiently small vectors $\Phi_{x}^{2}$ and $\Phi_{y}^{2}$ induce  admissible perturbations according to \eqref{second1}, but violate \eqref{second2}.
\end{proof}

With the notation \eqref{ab} and using \eqref{optcond3}, we can rewrite \eqref{nec1} as
\begin{align}
\label{optcond40} a\;X\geq t_0\,X+(1-t_0)\,Y\quad&\Longrightarrow\quad b\,Y\leq t_0\,X+(1-t_0)\,Y.
\end{align}

\begin{lemma}\label{optcond51}
For $0<t_0<1$ \eqref{optcond40} implies the property:
 $a=1$ if and only if $b=1$.
\end{lemma}
\begin{proof}
1. If $a=1$, then \eqref{optcond40} states
\begin{align*}
Y\leq X\quad&\Longrightarrow\quad Y\leq\frac{t_0}{b+t_0-1}\,X\qquad\text{and hence}\qquad X\leq \frac{t_0}{b+t_0-1}\,X.
\end{align*}
Since $X$ can be chosen to have either sign, this inequality is an equality and $b=1$.
\medskip

\noindent
2. If $b=1$, then \eqref{optcond40} states
\begin{align*}
Y\leq \frac{a-t_0}{1-t_0}\,X\quad\Longrightarrow\quad Y\leq X\qquad\text{and hence}\qquad \frac{a-t_0}{1-t_0}\,X\leq X.
\end{align*}
Varying the sign of $X$ as above, this inequality is an equality and $a=1$.
\end{proof}
\begin{lemma}\label{optcond5}
If $0<t_0<1$, then \eqref{optcond40} implies the  properties:
\begin{itemize}
\item[(i)] $b \ne 1-t_0$
\item[(ii)] If $a=b$ then $a=b=1$ and no further information on $t_0$ is available.
\item[(iii)] If $a\ne b$ then $t_0=a\,\frac{1-b}{a-b}$. Moreover either $b<1$ and $a>1$ or $b>1$ and $a<1$.
\end{itemize}
\end{lemma}
\begin{proof}
1. (i) is established by noting that $b=1-t_0$ in \eqref{optcond40} would yield the false implication
\begin{align*}
Y\leq X\, \frac{a-t_0}{1-t_0}\quad\Longrightarrow\quad X\geq 0.
\end{align*}
\medskip

\noindent
2. We will now establish the useful equality
\begin{align}\label{t0eq}
\frac{a-t_0}{1-t_0}\,=\,\frac{t_0}{b-(1-t_0)},
\end{align}
from which we will then deduce (ii) and (iii).
We proceed via a case distinction. Assuming first that
 $b>1-t_0$ and recalling $0<t_0<1$, we deduce from \eqref{optcond40} that
\begin{align*}
Y\leq X\, \frac{a-t_0}{1-t_0}\quad\Longrightarrow\quad Y\leq X\;\frac{t_0}{b-(1-t_0)}\qquad\text{and hence}\qquad X\, \frac{a-t_0}{1-t_0}\,\leq\,X\,\frac{t_0}{b-(1-t_0)}.
\end{align*}
Since $X$ can have any sign we obtain~\eqref{t0eq}.
In the case $b<1-t_0$ a similar argument yields the same equality \eqref{t0eq}.
\medskip

\noindent
3. To establish (ii), we substitute $a=b$ into \eqref{t0eq} and easily check that $a=b=1$ and $t_0\in(0,1)$ is not further specified. (We recall that $a\neq 0$.)
\medskip

\noindent
4. To establish (iii), we assume $a\ne b$ and solve for $t_0$ in \eqref{t0eq}  to obtain
\begin{align*}
t_0=a\,\frac{1-b}{a-b}.
\end{align*}
The condition $0<t_0<1$ then implies either $b<1$ and $a>1$ or $b>1$ and $a<1$.
\end{proof}

\paragraph{Contradiction to the optimality of $(x_0,y_0,z_0)$ in the case $t_0<1$.}\label{sec2}
Assume that $(x_0,y_0,z_0)$ is a positive extremal (cf.\ Definition~\ref{def:max}).
With Lemmas \ref{optcond1} -  \ref{optcond007} we can rewrite the admissibility condition \eqref{optcond0} as
\begin{align}\label{optcond01}
0&\overset{!}\leq
(t_k-t_0)\left\{ T_1-\Phi_{x}^{1}\cdot\nabla u(x_0)\right\}+
(t_k-t_0)^2\biggl\{
T_2-t_0(\Phi_{x}^2-\Phi_{y}^2)\cdot\xi_{0}\vert\nabla u(z_0)\vert\notag\\
&\qquad
-
\Phi_{y}^2\cdot\xi_{0}\vert\nabla u(z_0)\vert+\frac{1}{2}\Phi_{x}^{1}\cdot D^2u(x_0)\cdot\Phi_{x}^{1}
\Phi_{x}^{2}\cdot\nabla u(x_0)\biggr\}
+
O(\vert t_k-t_0\vert^3),
\end{align}
where
\begin{align*}
T_1&:=\vert\nabla u(z_0)\vert\left\{t_0\Phi_{x}^{1}\cdot\xi_{0}+(1-t_0)\Phi_{y}^{1}\cdot\xi_{0}\right\},\\
T_2&:=
-
\frac{t_0^2}{2}\left(\Phi_{x}^{1}-\Phi_{y}^{1}\right)\cdot D^2u(z_0)\cdot\left(\Phi_{x}^{1}-\Phi_{y}^{1}\right)
-
t_0\left(\Phi_{x}^{1}-\Phi_{y}^{1}\right)\cdot D^2u(z_0)\cdot\left(y_0-x_0\right)\\
\nonumber&
-
t_0\left(\Phi_{x}^{1}-\Phi_{y}^{1}\right)\cdot D^2u(z_0)\cdot\Phi_{y}^{1}
-
\left(y_0-x_0\right)\cdot D^2u(z_0)\cdot\Phi_{y}^{1}
-
\frac{1}{2}\Phi_{y}^{1}\cdot D^2u(z_0)\cdot\Phi_{y}^{1}\\
\nonumber&
+
\left(\Phi_{x}^{1}-\Phi_{y}^{1}\right)\cdot\nabla u(z_0).
\end{align*}
The necessary condition \eqref{maximality} for admissible perturbations can then be expressed in the form
\begin{align}\label{maximality2}
0&\geq (t_k-t_0)\left\{T_1-\Phi_{y}^{1}\cdot\nabla u(y_0)\right\}+
(t_k-t_0)^2\biggl\{T_2-t_0(\Phi_{x}^2-\Phi_{y}^2)\cdot\xi_{0}\vert\nabla u(z_0)\vert\notag\\
&\qquad-\Phi_{y}^2\cdot\xi_{0}\vert\nabla u(z_0)\vert+\frac{1}{2}\Phi_{y}^{1}\cdot D^2u(y_0)\cdot\Phi_{y}^{1}+\Phi_{y}^{2}\cdot\nabla u(y_0)\biggr\}
+
O(\vert t_k-t_0\vert^3).
\end{align}
Recalling the notation \eqref{ab}, we set
\begin{align}\label{pintro}
\Phi_{x}^{1}=\frac{b}{a}\,\Phi_{y}^{1}\qquad\hbox{and}\qquad\Phi_{y}^{1}=p\in B_{r}(0).
\end{align}
As a consequence we obtain
\begin{align*}
T_1-\Phi_{x}^{1}\cdot\nabla u(x_0)
&=
\vert\nabla u(z_0)\vert\left\{\frac{b}{a}\,t_0\,\Phi_{y}^{1}\cdot\xi_{0}+(1-t_0)\,\Phi_{y}^{1}\cdot\xi_{0}\right\}
-
\frac{b}{a}\Phi_{y}^{1}\cdot\nabla u(x_0)\\
&=
\vert\nabla u(z_0)\vert\left\{\frac{b}{a}\,t_0+(1-t_0)-b\right\}\Phi_{y}^{1}\cdot\xi_{0}.
\end{align*}
Recalling the two options (ii) and (iii) from Lemma \ref{optcond5}, we observe that $a=b=1$ results in
$T_1-\Phi_{x}^{1}\cdot\nabla u(x_0)=0$ and the same is true if $a\ne b$ (after simplification). Hence the first-order term in \eqref{optcond01} vanishes. The same reasoning leads to the cancellation of the first-order term in \eqref{maximality2}:
\begin{align*}
T_1-\Phi_{y}^{1}\cdot\nabla u(y_0)=0.
\end{align*}
\medskip

\noindent
With \eqref{pintro} $T_2$ becomes
\begin{align*}
T_2&=-\frac{t_0^2}{2}\,\left(\frac{b}{a}-1\right)^2\,p\cdot D^2u(z_0)\cdot p
-
t_0\,\left(\frac{b}{a}-1\right)\,p\cdot D^2u(z_0)\cdot (y_0-x_0)\\
&-
t_0\,\left(\frac{b}{a}-1\right)\,p\cdot D^2u(z_0)\cdot p
-
p\cdot D^2u(z_0)\cdot (y_0-x_0)
-
\frac{1}{2}\,p\cdot D^2u(z_0)\cdot p\\
&+
\left(\frac{b}{a}-1\right)\,p\cdot\nabla u(z_0).
\end{align*}
This is equivalent to
\begin{align*}
T_2&=
-
\frac{1}{2}\left(t_0\left[\frac{b}{a}-1\right]+1\right)^2\,p\cdot D^2u(z_0)\cdot p
-
\left(t_0\,\left[\frac{b}{a}-1\right]+1\right)\,p\cdot D^2u(z_0)\cdot (y_0-x_0)\\
\nonumber&
\qquad+
\left[\frac{b}{a}-1\right]p\cdot \xi_{0}\,\vert\nabla u(z_0)\vert.
\end{align*}
For $a\ne b$, the definition of $t_0$ yields
\begin{align}\label{T2p}
T_2=-\frac{b^2}{2}\,p\cdot D^2u(z_0)\cdot p
-
b\,p\cdot D^2u(z_0)\cdot (y_0-x_0)
+
\left[\frac{b}{a}-1\right]p\cdot \xi_{0}\,\vert\nabla u(z_0)\vert.
\end{align}
In case $a=b$ we obtain
\begin{align}\label{T2pa=b}
T_2=-\frac{1}{2}\,p\cdot D^2u(z_0)\cdot p
-
p\cdot D^2u(z_0)\cdot (y_0-x_0).
\end{align}
With this \eqref{optcond01} becomes
\begin{align}\label{optcond02}
0&\leq(t_k-t_0)^2
\biggl\{
T_2-(1-t_0)\,\Phi_{y}^{2}\cdot\xi_{0}\vert\nabla u(z_0)\vert
+
(a-t_0)\,\Phi_{x}^{2}\cdot\xi_{0}\vert\nabla u(z_0)\vert\\
\nonumber&+
\frac{1}{2}\frac{b^2}{a^2}\,p\cdot D^2u(x_0)\cdot p\biggr\}
+
O(\vert t_k-t_0\vert^3),
\end{align}
where we also used the fact that $\Phi_{x}^{2}\cdot\nabla u(x_0)=a\,\Phi_{x}^{2}\cdot\xi_{0}\vert\nabla u(z_0)\vert$.
\medskip

\noindent
Analogously \eqref{maximality2} becomes
\begin{align}\label{maximality3}
0&\geq(t_k-t_0)^2\biggl\{
T_2
+
(b+t_0-1)\Phi_{y}^{2}\cdot\xi_0\vert\nabla u(z_0)\vert
-
t_0\Phi_{x}^{2}\cdot\xi_{0}\vert\nabla u(z_0)\vert\\
\nonumber&+
\frac{1}{2}\,p\cdot D^2u(y_0)\cdot p
\biggr\}
+
O(\vert t_k-t_0\vert^3),
\end{align}
where $T_2$ is given in \eqref{T2p} for $a\ne b$  and in \eqref{T2pa=b} for $a=b$ .
\medskip

\noindent
For $p\in B_r(0)$ let $\epsilon(p)$ be a smooth function with $\epsilon(p)>0$ for $p\in B_{r}(0)\setminus\{0\}$ and $\epsilon(0)=0$. The precise expression for this function will be given later.
\medskip

\noindent
We choose $\Phi_{y}^{2}$ such that
\begin{align*}
\Phi_{y}^{2}\cdot\xi_{0}\vert\nabla u(z_0)\vert
=
\frac{1}{1-t_0}
\biggl\{
T_2
+
(a-t_0)\Phi_{x}^{2}\cdot\xi_{0}\vert\nabla u(z_0)\vert
+
\frac{1}{2}\,\frac{b^2}{a^2}\,p\cdot D^2u(x_0)\cdot p-\epsilon(p)
\biggr\}.
\end{align*}
Then \eqref{optcond02} and \eqref{maximality3}  become
\begin{align*}
0\leq
(t_k-t_0)^2\,
\epsilon(p)
+
O(\vert t_k-t_0\vert^3)\quad\hbox{for}\:p\in B_r(0)
\end{align*}
and
\begin{align*}
0&\geq(t_k-t_0)^2
\biggl\{
\frac{b}{1-t_0}\,T_2
+
\left(\left[\frac{b}{1-t_0}-1\right](a-t_0)-t_0\right)\Phi_{x}^{2}\cdot\xi_{0}\vert\nabla u(z_0)\vert\\
&+
\frac{1}{2}\,p\cdot D^2u(y_0)\cdot p
+
\frac{1}{2}\,\frac{b^2}{a^2}\left(\frac{b}{1-t_0}-1\right)\,p\cdot D^2u(x_0)\cdot p
-
\left(\frac{b}{1-t_0}-1\right)\,\epsilon(p)
\biggr\}\\
&+
O(\vert t_k-t_0\vert^3),
\end{align*}
respectively.
We compute the coefficient of $\Phi_{x}^{2}\cdot\xi_{0}\vert\nabla u(z_0)\vert$:
\begin{align*}
\left(\left[\frac{b}{1-t_0}-1\right](a-t_0)-t_0\right)
=
\frac{1}{1-t_0}\left(a(b-1)+(a-b)t_0\right).
\end{align*}
From Lemma \ref{optcond5} (ii) and (iii) we deduce that $\left(a(b-1)+(a-b)t_0\right)=0$ for any $a$ and $b$. Thus the maximality condition \eqref{maximality3} reduces to
\begin{align}\label{maxmaria4}
0&\geq(t_k-t_0)^2 W(p)
+
O(\vert t_k-t_0\vert^3)
\end{align}
for
\begin{align*}
W(p)&\coloneqq\frac{b}{1-t_0}\,T_2
+
\frac{1}{2}\,p\cdot D^2u(y_0)\cdot p
+
\frac{1}{2}\,\frac{b^2}{a^2}\left(\frac{b}{1-t_0}-1\right)\,p\cdot D^2u(x_0)\cdot p\\
&\qquad-
\left(\frac{b}{1-t_0}-1\right)\,\epsilon(p).
\end{align*}
Clearly $W(0)=0$ (which is also a consequence of \eqref{T2p}, \eqref{T2pa=b} and $\epsilon(0)=0$).
We will have succeeded in contradicting the existence of a positive extremal if we can find $r>0$ sufficiently small and any  function $\epsilon(p)$ that is strictly positive on $B_r(0)\setminus\{ 0\}$ and at the same time a
$p\in B_r(0)\setminus\{ 0\}$ such that $W(p)>0$, since we will then have found admissible variations that violate \eqref{maxmaria4}.
\medskip

\noindent
We distinguish between $a\ne b$ and $a=b=1$.
\medskip

\noindent
The Case $a\ne b$: The explicit formula $t_0=a\,\frac{1-b}{a-b}$ (see Lemma \ref{optcond5} (iii)) yields
\begin{align*}
\frac{b}{1-t_0}=\frac{a-b}{a-1}>0\quad\hbox{and}\quad\frac{b}{1-t_0}-1=\frac{1-b}{a-1}>0,
\end{align*}
since either $a<1$ and $b>1$ or $a>1$ and $b<1$.
\medskip

\noindent
Let $\Delta_{p}$ be the Laplace operator, where the differentiation is with respect to the $p$ - variable.
With \eqref{T2p} we compute
\begin{align*}
\Delta_{p}T_2=-b^2\,\Delta u(z_0)
\end{align*}
Thus
\begin{align*}
\Delta_{p}W(p)
=
-b^2\frac{a-b}{a-1}\,\Delta u(z_0)
+
\Delta u(y_0)
+
\frac{b^2}{a^2}\,\frac{1-b}{a-1}\Delta u(x_0)
-
\frac{1-b}{a-1}\Delta_{p}\epsilon(p).
\end{align*}
Since $u$ solves $\Delta u=f(u)$ and since $u(x_0)=u(z_0)$ (see Lemma \ref{optcond1}) we obtain
\begin{align}
\Delta_{p}W(p)
&=&
f(u(y_0))-f(u(z_0))
+
f(u(z_0))\left\{1-b^2\,\frac{a-b}{a-1}+\frac{b^2}{a^2}\,\frac{1-b}{a-1}\right\}
-
\frac{1-b}{a-1}\,\Delta_{p}\epsilon(p)\notag\\
&=&
f(u(y_0))-f(u(z_0))
+
\frac{1}{a^2}f(u(z_0))\,(1-b)(a-b)(ab+a+b)
-
\frac{1-b}{a-1}\,\Delta_{p}\epsilon(p).\label{delp}
\end{align}
Observe that
\begin{align*}
(1-b)(a-b)(ab+a+b)>0\qquad\hbox{and}\qquad\frac{1-b}{a-1}>0.
\end{align*}
Since we assume \eqref{convexass2} and {\bf{(H2)}} we deduce $f(u(y_0))-f(u(z_0))>0$ and $f(u(z_0))>0$ and hence
\begin{align}\label{H2as}
f(u(y_0))-f(u(z_0))
+
\frac{1}{a^2}f(u(z_0))\,(1-b)(a-b)(ab+a+b)=:\gamma_0>0.
\end{align}
Substituting \eqref{H2as} together with
\begin{align*}
\epsilon(p):=\frac{\gamma_0}{4n}\,\frac{a-1}{1-b}\,\vert p\vert^2
\end{align*}
into \eqref{delp} yields
\begin{align*}
W(0)=0\qquad\hbox{and}\qquad\Delta_p W(p)=\frac{\gamma_0}{2}>0.
\end{align*}
It follows that $\tilde{W}(p):=W(p)-\frac{\gamma_0}{4n}\vert p\vert^2$ is subharmonic in $B_r(0)$.
The mean value formula for subharmonic functions then implies
\begin{align*}
\sup_{p\in B_r(0)}W(p)\geq \frac{\gamma_0}{4n(n+2)} r^2.
\end{align*}
Hence we have obtained the positivity of $W(p)$ and thus a contradiction.
\medskip

\noindent
The Case $a=b=1$: In this case $\frac{b}{1-t_0}-1=\frac{t_0}{1-t_0}>0$.
Proceeding as before we obtain
\begin{align*}
W(p):=\frac{1}{1-t_0}\,T_2
+
\frac{1}{2}\,p\cdot D^2u(y_0)\cdot p
+
\frac{1}{2}\left(\frac{t_0}{1-t_0}\right)\,p\cdot D^2u(x_0)\cdot p-
\left(\frac{t_0}{1-t_0}\right)\,\epsilon(p).
\end{align*}
With \eqref{T2pa=b} we compute
\begin{align*}
\Delta_{p}T_2=-\Delta u(z_0).
\end{align*}
Then
\begin{align*}
\Delta_{p}W(p)=f(u(y_0))-f(u(z_0))-\frac{t_0}{1-t_0}\Delta_{p}\epsilon(p).
\end{align*}
This leads as above to the existence of $p$ such that $W(p)>0$.
Note that if $a=b=1$, then the assumption $f(c_{*})\geq 0$ in (H2) is not necessary.
\subsubsection{The case $t_0=1$}\label{sec3}
In case $t_0=1$ ($x_0=z_0$), the argument proceeds somewhat differently. We begin with the proof of Lemma \ref{optcond007}. Then we give an auxiliary lemma (Lemma \ref{optcond12}) and use these together to establish Lemma \ref{optcond22}. Finally we include a higher order optimality condition (Lemma \ref{optcond13}) which will be necessary for our contradiction argument in this case.

We recall the expansions \eqref{optcond0}, \eqref{maximality} generated by \eqref{xk}, \eqref{yk}. In the case $t_0=1$, we will need to also keep the third order term,  so that the expansion (in light of $t_0=1$ ($x_0=z_0$), $u(x_0)=u(z_0)$, and $(y_0-x_0)\cdot \nabla u(x_0)=0$) takes the form:
\begin{align}\label{optcond16}
0 &\overset{!}{<} u(x_k)-u(z_k)\notag\\
&= (1-t_k)^2\biggl\{\left(\Phi_{x}^{1}-\Phi_{y}^{1}\right)\cdot\nabla u(z_0)
-
\frac{1}{2}(y_0-z_0)\cdot D^2u(z_0)\cdot(y_0-z_0)-\Phi_{x}^{1}\cdot D^2u(z_0)\cdot(y_0-z_0)
\biggr\}\notag\\
\nonumber&
-
(1-t_k)^3\Biggl\{
(y_0-z_0)\cdot
\biggl[\frac{1}{6}(y_0-z_0)\cdot D^3u(z_0)\cdot(y_0-z_0)
+
\frac{1}{2}\Phi_{x}^{1}\cdot D^3u(z_0)\cdot\Phi_{x}^{1}\\
&+
\frac{1}{2}\Phi_{x}^{1}\cdot D^3u(z_0)\cdot(y_0-z_0)
\biggr]
-
\Phi_{x}^{1}\cdot D^2u(z_0)\cdot\Phi_{x}^{1}
-
(y_0-z_0)\cdot D^2u(z_0)\cdot\Phi_{x}^{1}\\
\nonumber&
+
\Phi_{x}^{1}\cdot D^2u(z_0)\cdot\Phi_{y}^{1}
+
(y_0-z_0)\cdot D^2u(z_0)\cdot\Phi_{x}^{2}
+
(y_0-z_0)\cdot D^2u(z_0)\cdot\Phi_{y}^{1}
+
\Phi_{y}^{2}\cdot \nabla u(z_0)\\
\nonumber&
-
\Phi_{x}^{2}\cdot \nabla u(z_0)
\Biggr\}
+O\left((1-t_k)^4\right).
\end{align}
In that case there holds the necessary condition
\begin{align}
\label{optcond7}0&\geq u(y_k)-u(z_k)-(u(y_0)-u(z_0))\\
&=
\nonumber(1-t_k)\biggl\{\Phi_{y}^{1}\cdot\nabla u(y_0)-\Phi_{x}^{1}\cdot\nabla u(z_0)
\biggr\}\\
\nonumber&+
(1-t_k)^2\biggl\{-\frac{1}{2}(y_0-z_0)\cdot D^2u(z_0)\cdot(y_0-z_0)
-
\Phi_{x}^1\cdot D^2u(z_0)\cdot(y_0-z_0)\\
\nonumber&
-
\frac{1}{2}\Phi_{x}^1\cdot D^2u(z_0)\cdot\Phi_{x}^1
+
\left(\Phi_{x}^{1}-\Phi_{y}^{1}\right)\cdot\nabla u(z_0)
-
\Phi_{x}^2\cdot\nabla u(z_0)
+
\frac{1}{2}\Phi_{y}^1\cdot D^2u(y_0)\cdot\Phi_{y}^1\\
\nonumber&
+
\Phi_{y}^{2}\cdot\nabla u(y_0)
\biggr\}
+O\left((1-t_k)^3\right).
\end{align}
\begin{proof}[Proof of Lemma~\ref{optcond007} for $t_0=1$]
\,\newline
If $(y_0-z_0)\cdot D^2u(z_0)\cdot(y_0-z_0)<0$, then we can choose $\Phi_x^1=0$, $\Phi_y^1\perp\nabla u(z_0)$ and $\Phi_{xy}^2$ sufficiently small
such that \eqref{optcond16} is satisfied (the perturbed points are admissible) but \eqref{optcond7} is violated.

If on the other hand $(y_0-z_0)\cdot D^2u(z_0)\cdot(y_0-z_0)>0$, then we can choose $\Phi_x^1=-(y_0-z_0)$ and
\begin{align*}
\Phi_{y}^{1}:=\delta\,\frac{\nabla u(y_0)}{\vert\nabla u(y_0)\vert^2}
\end{align*}
for $\delta>0$ small enough so that \eqref{optcond16} is satisfied (the perturbed points are admissible) but \eqref{optcond7} is violated.
Thus $(y_0-z_0)\cdot D^2u(z_0)\cdot(y_0-z_0)=0$.
\end{proof}
\begin{lemma}\label{optcond12}
For any positive extremal $(x_0,y_0,z_0)$ with $t_0=1$ we have
$\nabla u(z_0)=\alpha\,D^2u(z_0)\cdot(y_0-z_0)$ for some $\alpha\in\R\setminus\{0\}$.
\end{lemma}
\begin{proof}
We assume there is no $\alpha\in\R\setminus\{0\}$ such that $\nabla u(z_0)=\alpha\,D^2u(z_0)\cdot(y_0-z_0)$, so that $\nabla u(z_0)$ and $D^2 u(z_0)\cdot (y_0-z_0)$ are linearly independent.
With Lemma  \ref{optcond007}, we write \eqref{optcond16} as
\begin{align}
0&\overset{!}<u(x_k)-u(z_k)=
(1-t_k)^2\biggl\{
\left(\Phi_{x}^{1}-\Phi_{y}^{1}\right)\cdot\nabla u(z_0)
-
\Phi_{x}^{1}\cdot D^2u(z_0)\cdot(y_0-z_0)
\biggr\}\\
&
+
O((1-t_k)^3).\label{437b}
\end{align}
Since $\nabla u(z_0)$ and $D^2 u(z_0)\cdot (y_0-z_0)$  are linearly independent, we can choose
$\Phi_{x}^{1}$ such that $\Phi_{x}^{1}\cdot\nabla u(z_0)=0$ and
$\Phi_{x}^{1}\cdot D^2u(z_0)\cdot(y_0-z_0)\ne 0$; in particular we can choose $\Phi_x^1$ such that
\begin{align*}
\Phi_{x}^{1}\cdot D^2u(z_0)\cdot(y_0-z_0)<0.
\end{align*}
Let $\Phi_{y}^{1}=\epsilon\,\nabla u(y_0)$ for $\epsilon>0$. Then
\begin{align*}
&\left(\Phi_{x}^{1}-\Phi_{y}^{1}\right)\cdot\nabla u(z_0)
-
\Phi_{x}^{1}\cdot D^2u(z_0)\cdot(y_0-z_0)\\
&\qquad=
-\epsilon\,\nabla u(y_0)\cdot\nabla u(z_0)
-
\Phi_{x}^{1}\cdot D^2u(z_0)\cdot(y_0-z_0)
>0
\end{align*}
for sufficiently small $\epsilon$, so that according to \eqref{437b}, the sequence is admissible.
The necessary condition \eqref{optcond7} then implies
\begin{align*}
0\geq\Phi_{y}^{1}\cdot\nabla u(y_0)-\Phi_{x}^{1}\cdot\nabla u(z_0)=\epsilon\,\vert\nabla u(y_0)\vert^2.
\end{align*}
This contradiction concludes the proof.
\end{proof}
Taking this into account we may rewrite \eqref{optcond16} and \eqref{optcond7} as
\begin{align}
\label{optcond8}
0\overset{!}<u(x_k)-u(z_k)&=(1-t_k)^2\biggl\{\left((1-\alpha^{-1})\Phi_{x}^{1}-\Phi_{y}^{1}\right)\cdot\nabla u(z_0)
\biggr\}+O\left((1-t_k)^3\right)
\end{align}
and
\begin{align}\label{optcond9}
0&\geq
(1-t_k)\biggl\{\Phi_{y}^{1}\cdot\nabla u(y_0)-\Phi_{x}^{1}\cdot\nabla u(z_0)
\biggr\}
+O\left((1-t_k)^2\right).
\end{align}
\begin{proof}[Proof of Lemma \ref{optcond22} in the case $t_0=1$]
$ $\newline
Choose any vector $\Phi_{x}^{1}$ with $\Phi_{x}^{1}\cdot\nabla u(z_0)=0$. Then for $k\geq k_0$ suffiently large \eqref{optcond8} $\Longrightarrow$ \eqref{optcond9} is equivalent to
\begin{align*}
-\Phi_{y}^{1}\cdot\nabla u(z_0)>0\quad\Longrightarrow\quad\Phi_{y}^{1}\cdot\nabla u(y_0)\leq 0.
\end{align*}
This can only be true if $\nabla u(z_0)$ and $\nabla u(y_0)$ point in the same direction.
As a consequence we can formulate \eqref{optcond8} $\Longrightarrow$ \eqref{optcond9} as:
\medskip

\noindent
Let $\Phi_{x}^{1},\Phi_{y}^{1}$ be any perturbations, satisfying
\begin{align}\label{nec20}
0<(1-t_k)^2\left\{\left(1-\frac{1}{\alpha}\right)\Phi_{x}^{1}\cdot\xi_{0}-\Phi_{y}^{1}\cdot\xi_{0}\right\}
+
O((1-t_k)^3),
\end{align}
then necessarily the inequality
\begin{align}\label{nec21}
0\geq \Phi_{y}^{1}\cdot\xi_{0}\,\vert\nabla u(y_0)\vert-\Phi_{x}^{1}\cdot\xi_{0}\,\vert\nabla u(z_0)\vert
+
O((1-t_k))
\end{align}
holds. Recalling the notation \eqref{ab},
the implication \eqref{nec20} $\Longrightarrow$ \eqref{nec21} can be expressed
\begin{align*}
Y<\left(1-\frac{1}{\alpha}\right)X\quad\Longrightarrow\quad bY\leq X\qquad\text{and hence}\qquad
\left(1-\frac{1}{\alpha}\right)\,b\,X\leq X.
\end{align*}
Since $X$ can have any sign we obtain
\begin{align}\label{b}
\left(1-\frac{1}{\alpha}\right)=\frac{1}{b}\qquad\text{and}\quad b\ne 1.
\end{align}
\end{proof}
\begin{lemma}\label{optcond13}
For any positive extremal $(x_0,y_0,z_0)$ with $t_0=1$ we have
\begin{align*}
(y_0-z_0)\cdot\left[(y_0-z_0)\cdot D^3u(z_0)\cdot(y_0-z_0)\right]= 0.
\end{align*}
\end{lemma}
\begin{proof}
1. Assume $(y_0-z_0)\cdot\left[(y_0-z_0)\cdot D^3u(z_0)\cdot(y_0-z_0)\right]<0$.
We set $\Phi_{x}^{1}=(y_0-z_0)$ and $\Phi_{y}^{1}=0$. With this choice and Lemmas \ref{optcond2} and \ref{optcond007}, the admissibility condition \eqref{optcond16} takes the form
\begin{align}
0&\overset{!}<u(x_k)-u(z_k)=-(1-t_k)^3\Biggl\{\frac{7}{6}(y_0-z_0)\cdot\left[(y_0-z_0)\cdot D^3u(z_0)\cdot(y_0-z_0)\right]\notag\\
&\quad+
(y_0-z_0)\cdot D^2u(z_0)\cdot\Phi_{x}^{2}+\Phi_{y}^{2}\cdot\nabla u(z_0)-\Phi_{x}^{2}\cdot\nabla u(z_0)
\Biggr\}
+
O((1-t_k)^4).\label{star}
\end{align}
With Lemmas \ref{optcond12}-\ref{optcond22} and \eqref{b} we compute
\begin{align}
\lefteqn{(y_0-z_0)\cdot D^2u(z_0)\cdot\Phi_{x}^{2}+\Phi_{y}^{2}\cdot\nabla u(z_0)-\Phi_{x}^{2}\cdot\nabla u(z_0)}\notag\\
&\qquad=
\left(\frac{1}{\alpha}\Phi_{x}^{2}\cdot\xi_{0}+\Phi_{y}^{2}\cdot\xi_{0}-\Phi_{x}^{2}\cdot\xi_{0}\right)\vert\nabla u(z_0)\vert=
\vert\nabla u(z_0)\vert\left[-\frac{1}{b}\Phi_{x}^{2}\cdot\xi_0+\Phi_{y}^{2}\cdot\xi_0\right].\label{comp}
\end{align}
Substituting \eqref{comp} into \eqref{star} yields the admissibility condition
\begin{align}\label{contra0}
\frac{7}{6}(y_0-z_0)\cdot\left[(y_0-z_0)\cdot D^3u(z_0)\cdot(y_0-z_0)\right]
+
\vert\nabla u(z_0)\vert\left[-\frac{1}{b}\Phi_{x}^{2}\cdot\xi_0+\Phi_{y}^{2}\cdot\xi_0\right]\leq 0
\end{align}
on $\Phi_{x}^{2}$ and on $\Phi_{y}^{2}$.
On the other hand the necessary condition \eqref{optcond7} becomes
\begin{align}
-\frac{1}{b}\Phi_{x}^{2}\cdot\xi_0+\Phi_{y}^{2}\cdot\xi_0\leq 0.\label{contra1}
\end{align}

Choosing $\Phi_{x}^{2}$ and $\Phi_{y}^{2}$ such that $-\frac{1}{b}\Phi_{x}^{2}\cdot\xi_0+\Phi_{y}^{2}\cdot\xi_0>0$ and $\vert\Phi_{x,y}^{2}\vert<r$ small enough such that \eqref{contra0} holds
leads to a contradiction.

\medskip

\noindent
2. Now assume $(y_0-z_0)\cdot\left[(y_0-z_0)\cdot D^3u(z_0)\cdot(y_0-z_0)\right]>0$. Using this assumption together with Lemmas
\ref{optcond2} and \ref{optcond007} in a Taylor-expansion for small, positive $s$, we deduce
\begin{align*}
  u(z_0-s(y_0-z_0))-u(z_0)<0.
\end{align*}
We want to use this to contradict optimality of $(x_0,y_0,z_0)$, since for
\begin{align}
  \tilde{x}_{0}=\tilde{z}_{0}=x_0-s(y_0-z_0),\label{xtil}
\end{align}
the  triple $(\tilde{x}_{0},y_0,\tilde{x}_{0})$ satisfies $u(y_0)-u(\tilde{z}_{0})>u(y_0)-u(z_0)$.

\medskip

\noindent

To this end, it suffices to show that $(\tilde{x}_{0},y_0,\tilde{x}_{0})$ is admissible, i.e., that there exists a sequence $({x}_{k},y_0,{z}_{k})$ such that $u({x}_{k})>u({z}_{k})$ and
\begin{align*}
\lim_{k\to\infty}({x}_{k},y_0,{z}_{k})=(\tilde{x}_{0},y_0,\tilde{x}_{0}).
\end{align*}
For this we set $\Phi_{y}^{1}=\Phi_{y}^{2}=\Phi_{x}^{2}=0$ and $\Phi_{x}^{1}= \xi_{0}$ and for a sequence $t_k\uparrow 1$ define
 $(x_k,y_0,z_k)$ as
\begin{align*}
x_k&=\tilde{x}_0+(1-t_k)\, \xi_{0},\\
z_k&=t_k x_k+(1-t_k)y_0=\tilde{x}_0 + (1-t_k)\,( \xi_{0} + y_0-\tilde{x}_0)-(1-t_k)^{2} \,\xi_{0}.
\end{align*}

Expanding jointly in $(1-t_k)$ and $s$ (recall \eqref{xtil}) and using  Lemmas \ref{optcond2} and \ref{optcond007}-\ref{optcond22}, we obtain
\begin{align*}
u(\tilde{x}_{k})-u(\tilde{z}_{k})=
\frac{s^3}{6}(y_0-z_0)\cdot[(y_0-z_0)\cdot D^3u(z_0)\cdot(y_0-z_0)]+O(s^4)+O(s(1-t_k)),
\end{align*}
so that for $s$ small and positive and $t_k$ sufficiently close to $1$, the sequence
$(\tilde{x}_{k},y_0,\tilde{z}_{k})$ is admissible.
\end{proof}
\paragraph{Contradiction to the optimality of $(x_0,y_0,z_0)$ in the case $t_0=1$.}\label{sec4}

Assume that $(x_0,y_0,z_0)$ is a positive extremal (cf. Definition~\ref{def:max}).
In a first step we apply Lemmas \ref{optcond2} - \ref{optcond13} and \eqref{b} to the inequalities \eqref{optcond16} and \eqref{optcond7}. Defining
\begin{align*}
T_1&:=\vert\nabla u(z_0)\vert\left(\frac{1}{b}\Phi_{x}^{1}\cdot\xi_{0}-\Phi_{y}^{1}\cdot\xi_{0}\right)\\
T_2&:= -(y_0-z_0)\cdot\left[\frac{1}{2}\Phi_{x}^{1}\cdot D^3u(z_0)\cdot\Phi_{x}^{1}+\frac{1}{2}\Phi_{x}^{1}\cdot D^3u(z_0)\cdot(y_0-z_0)\right]\\
T_3&:=\vert\nabla u(z_0)\vert\left(\Phi_{x}^{1}\cdot\xi_{0}-\Phi_{y}^{1}\cdot\xi_{0}\right)\\
T_4&:=\vert\nabla u(z_0)\vert\left(\frac{1}{b}\Phi_{x}^{2}\cdot\xi_{0}-\Phi_{y}^{2}\cdot\xi_{0}\right),
\end{align*}
we can write the admissibility condition \eqref{optcond16} in the form
\begin{align}
\label{new16} 0&\overset{!}{<}u(x_k)-u(z_k)=
(1-t_k)^2\,T_1\\
\nonumber&+
(1-t_k)^3\left\{T_2
+
\Phi_{x}^{1}\cdot D^2u(z_0)\cdot\Phi_{x}^{1}
-
\Phi_{x}^{1}\cdot D^2u(z_0)\cdot\Phi_{y}^{1}
+
\frac{1}{\alpha}T_3+T_4\right\}
+
O((1-t_k)^4)
\end{align}
and the necessary condition  \eqref{optcond7} as
\begin{align}\label{new17}
0&\geq u(y_k)-u(z_k)-(u(y_0)-u(z_0))=
-(1-t_k)\,b\,T_1\\
\nonumber&+
(1-t_k)^2\left\{
\frac{1}{b}\,T_3
-
\frac{1}{2}\Phi_{x}^{1}\cdot D^2u(z_0)\cdot\Phi_{x}^{1}
+
\frac{1}{2}\Phi_{y}^{1}\cdot D^2u(y_0)\cdot\Phi_{y}^{1}
-
b\, T_4\right\}
+
O((1-t_k)^3).
\end{align}
We will be done if we can show that there exist perturbations $\Phi_{x,y}^{1}$ and $\Phi_{x,y}^{2}$ that satisfy \eqref{new16} but fail to satisfy \eqref{new17}.
\medskip

\noindent
Let $r>0$. For $p\in B_r(0)$  we set
\begin{align*}
\Phi_{x}^{1}:=b\,p\qquad\Phi_{y}^{1}:=p.
\end{align*}
Then $T_1=0$ and
\begin{align*}
T_2&=
 -(y_0-z_0)\cdot\left[\frac{b^2}{2}p\cdot D^3u(z_0)\cdot p+\frac{b}{2}p\cdot D^3u(z_0)\cdot(y_0-z_0)\right],\\
T_3&= (b-1)\,\vert\nabla u(z_0)\vert\,p\cdot\xi_{0}.
\end{align*}
Next we choose $\Phi_{x}^{2}$ and $\Phi_{y}^{2}$ such that
\begin{align*}
T_4
=
\epsilon(p)
-
T_2
-
b^2\,p\cdot D^2u(z_0)\cdot p
+
b\,p\cdot D^2u(z_0)\cdot p
-
\frac{1}{\alpha}T_3
\end{align*}
for some function $\epsilon(p)$ with $\epsilon(0)=0$ and $\epsilon(p)>0$ for $p\in B_{r}(0)\setminus\{0\}$. Then for $p\ne 0$ the perturbed sequence satisfies the admissibility condition \eqref{new16}, since
\begin{align*}
0<u(x_k)-u(z_k)=
(1-t_k)^3\epsilon(p)
+
O((1-t_k)^4).
\end{align*}
\medskip

\noindent
In the next step we insert the expression for $T_4$ into \eqref{new17}. Then maximality of $(x_0,y_0,z_0)$ implies that the following inequality must hold:
\begin{align*}
0&\geq u(y_k)-u(z_k)-(u(y_0)-u(z_0))\\
&=
(1-t_k)^2\left\{
-b\,\epsilon(p)
+
\left(\frac{1}{b}+b-1\right)\,T_3
+
b\,T_2
+
\left(b^3-\frac{3}{2}b^2\right)\,p\cdot D^2u(z_0)\cdot p
+
\frac{1}{2}p\cdot D^2u(y_0)\cdot p\right\}\\
&\qquad+
O((1-t_k)^3).
\end{align*}
We will be done if we can show that this last inequality does not hold.
\medskip

\noindent
Let
\begin{align*}
W(p):=-b\,\epsilon(p)
+
\left(\frac{1}{b}+b-1\right)\,T_3
+
b\,T_2
+
\left(b^3-\frac{3}{2}b^2\right)\,p\cdot D^2u(z_0)\cdot p
+
\frac{1}{2}p\cdot D^2u(y_0)\cdot p.
\end{align*}
Clearly $W(0)=0$. We will compute $\Delta_{p}W(p)$:
\medskip

\noindent
Since $T_3$ depends linearily on $p$ we obtain $\Delta_{p} T_3=0$. Moreover
\begin{align*}
\Delta_{p} \left(p\cdot D^2u(z_0)\cdot p\right)=2\,\Delta u(z_0)=2\,f(u(z_0)),\quad
\frac{1}{2}\Delta_{p} \left(p\cdot D^2u(y_0)\cdot p\right)=\Delta u(y_0)=f(u(y_0)),
\end{align*}
and
\begin{align*}
\Delta_{p} T_2=-b^2(y_0-z_0)\cdot \nabla u(z_0)=0
\end{align*}
by Lemma \ref{optcond2}. Thus
\begin{align*}
\Delta_{p}W(p)=-b\Delta_{p}\epsilon(p)+\left(2b^3-3b^2+1\right)\,f(u(z_0))+f(u(y_0))-f(u(z_0)).
\end{align*}
We now define $\epsilon(p)$ as
\begin{align*}
\epsilon(p):=\frac{\epsilon_0}{2n\,b}\vert p\vert^2.
\end{align*}
Note that $2b^3-3b^2+1\geq0$ with only one zero in $b=1$ (and recall according to \eqref{b} that $b\neq 1$). Since $f(u(y_0))-f(u(z_0))>0$ and since $f(u(z_0))\geq 0$ by {\bf{(H2)}} we obtain
\begin{align*}
\Delta_{p}W(p)=-\gamma_{0}+f(u(y_0))-f(u(z_0))>0
\end{align*}
if $\gamma_{0}$ is chosen small enough. For such a choice we set $\gamma:=-\gamma_{0}+f(u(y_0))-f(u(z_0))>0$. Hence, $W(p)-\frac{\gamma}{2n}\vert p\vert^2$ is subharmonic in $B_r(0)$ and the mean value formula for subharmonic functions
implies
\begin{align*}
\sup_{p\in B_r(0)}W(p)\geq c(\gamma)r^2.
\end{align*}
This is the desired contradiction.

\medskip

\noindent


\section*{Acknowledgements}
We thank Eric Vanden--Eijnden for helpful discussions and for encouraging our interaction. Sebastian Scholtes was partially supported by DFG Grant WE 5760/1-1.

\end{document}